\pdfoutput=1
\documentclass[12pt]{amsart}
\usepackage{amssymb}
\usepackage[noadjust]{cite}
\usepackage{booktabs}
\usepackage{url}
\usepackage{hyphenat}
\usepackage{epsf}
\usepackage{mathtools}
\usepackage{mathrsfs}
\usepackage{array}
\usepackage{graphicx}
\usepackage[curve]{xypic}
\usepackage{enumitem}
\usepackage[pdftex,lmargin=1in,rmargin=1in,tmargin=1in,bmargin=1in]{geometry}
\usepackage[bookmarks=true, bookmarksopen=true,%
bookmarksdepth=3,bookmarksopenlevel=2,%
colorlinks=true,%
linkcolor=blue,%
citecolor=blue,%
filecolor=blue,%
menucolor=blue,%
urlcolor=blue]{hyperref}
\hypersetup{pdftitle={Noncommutative Hodge-to-de Rham spectral sequence and the Heegaard Floer homology of double covers}}
\hypersetup{pdfauthor={Robert Lipshitz and David Treumann}}

\newread\testin

\makeatletter
\newcommand\mi@kern[1]{%
  \settowidth\@tempdima{$\mi@obj^{#1}$}
  \kern-\@tempdima
  #1
  \settowidth\@tempdima{$\mi@obj$}
  \kern\@tempdima
}

\newtoks\mi@toksp
\newtoks\mi@toksb
\DeclareRobustCommand{\manyindices}[5]{
  \def\mi@obj{#5}
  \mi@toksp\expandafter{\mi@kern{#2}}
  \mi@toksb\expandafter{\mi@kern{#1}}
  \@mathmeasure4\textstyle{#5_{#1}^{#2}}
  \@mathmeasure6\textstyle{#5_{#3}^{#4}}
  \dimen0-\wd6 \advance\dimen0\wd4
  \@mathmeasure8\textstyle{\hphantom{{}_{#1}^{#2}}#5^{\the\mi@toksp#4}_{\the\mi@toksb#3}}
  \hbox to \dimen0{}{\kern-\dimen0\box8}
}
\makeatother 

\usepackage{ifpdf}
\ifpdf
  
\else
  
\fi

\newcommand{\RR}{\mathbb R}

\newcommand{\ZZ}{\mathbb Z}

\newcommand{\FF}{\mathbb F}

\newcommand{\Image}{\mathrm{Im}}
\newcommand{\bD}{\mathbb{D}}

\newcommand{\co}{\nobreak\mskip2mu\mathpunct{}\nonscript
  \mkern-\thinmuskip{:}\penalty300\mskip6muplus1mu\relax}

\newcommand{\bdy}{\partial}

\newcommand{\lbracket}{[}
\newcommand{\rbracket}{]}

\newcommand{\spinc}{\mathfrak s}
\newcommand{\spinct}{\mathfrak t}
\DeclareMathOperator{\divis}{div}

\DeclareMathOperator{\Hom}{Hom}

\DeclareMathOperator{\Ext}{Ext}
\DeclareMathOperator{\Tor}{Tor}

\DeclareMathOperator{\spin}{spin}
\newcommand{\SpinC}{\spin^c}

\DeclareMathOperator{\supp}{supp}

\DeclareMathOperator{\Fuk}{Fuk}

\newcommand{\Barop}{{\mathrm{Bar}}}

\newcommand\interior{\mathrm{int}}

\theoremstyle{plain}
\newtheorem{theorem}{Theorem}
\newtheorem{thmcor}[theorem]{Corollary} 
\numberwithin{equation}{section}
\newtheorem{citethm}[equation]{Theorem}
\newtheorem{proposition}[equation]{Proposition}
\newtheorem{lemma}[equation]{Lemma}
\newtheorem{corollary}[equation]{Corollary}

\newtheorem{conjecture}{Conjecture}

\newtheorem{definition}[equation]{Definition}

\newtheorem{citeconj}[equation]{Conjecture}

\theoremstyle{definition}

\newtheorem{question}[conjecture]{Question}

\theoremstyle{remark}
\newtheorem{example}[equation]{Example}
\newtheorem{remark}[equation]{Remark}

\newcommand{\HF}{\mathit{HF}}
\newcommand{\HFa}{\widehat {\HF}}

\newcommand{\CFa}{\widehat {\mathit{CF}}}
\newcommand{\HFKa}{\widehat{\mathit{HFK}}}
\newcommand{\x}{\mathbf x}

\newcommand\CH{\mathit{CH}}

\newcommand\HH{\mathit{HH}}
\newcommand\HC{\mathit{HC}}
\newcommand\Hochschild\HH

\newcommand{\bLam}{\boldsymbol{\Delta}C}
\newcommand{\TT}{\Pi}

\newcommand{\Ainf}{A_\infty}
\newcommand{\Alg}{\mathcal{A}}

\newcommand\Blg{\mathcal{B}}

\newcommand{\alphas}{{\boldsymbol{\alpha}}}
\newcommand{\betas}{{\boldsymbol{\beta}}}

\newcommand{\cM}{\mathcal{M}}

\newcommand{\cN}{\mathcal{N}}

\newcommand{\DD}{\textit{DD}}

\newcommand{\CFDA}{\mathit{CFDA}}
\newcommand{\CFDAa}{\widehat{\CFDA}}

\newcommand{\CFL}{\mathit{CFL}}
\newcommand{\CFLa}{\widehat{\CFL}}

\newcommand{\cZ}{\mathcal{Z}}
\newcommand{\PtdMatchCirc}{\cZ}
\newcommand{\PMC}{\PtdMatchCirc}

\newcommand{\dg}{dg }

\newcommand\Id{\mathbb{I}}
\newcommand\Ground{\mathbf k}

\newcommand\DTP{\otimes^L}
\newcommand\DT{\boxtimes}

\newcommand{\Field}{{\FF_2}}
\DeclareMathOperator{\nbd}{nbd}

\newcommand{\Heegaard}{\mathcal{H}}
\newcommand{\HD}{\Heegaard}

\newcommand{\cE}{\mathcal{E}}
\newcommand{\cF}{\mathcal{F}}
\renewcommand{\th}{^\text{th}}

\newcommand{\ModCat}{\mathsf{Mod}}

\newcommand{\SmallBar}{\lsup{\Alg}\textit{bar}^{\Alg}}

\newcommand{\op}{\mathrm{op}}
\newcommand\PunctF{F^\circ}
\newcommand{\Chord}{\mathrm{Chord}}

\makeatletter
\newcommand\honestalg[3]{\bigl\lbracket
\begin{smallmatrix} #1\@ifempty{#3}{}{&#3} \\ #2 \end{smallmatrix}
\bigr\rbracket}

\makeatother

\newcommand{\sos}[3]{\mathbin{{}_{#1}\mathord#2_{#3}}}
\newcommand{\lsub}[2]{{}_{#1}#2}
\newcommand{\lsup}[2]{{}^{#1}\mskip-.6\thinmuskip#2}

\newcommand{\Cyc}{\mathrm{Cyc}}

\newcommand{\vhE}{\lsup{vh}E}
\newcommand{\hvE}{\lsup{hv}E}

\newcommand{\hF}{\lsup{h}F}
\newcommand{\vF}{\lsup{v}F}

\newcommand{\vhd}{\lsup{vh}d}
\newcommand{\hvd}{\lsup{hv}d}

\newcommand{\arcz}{\mathbf{z}}
\newcommand{\TC}{\mathsf{TC}}
\newcommand{\TCc}{\TC_0}
\newcommand{\SFH}{\mathit{SFH}}
\newcommand{\BSD}{\mathit{BSD}}
\newcommand{\BSDa}{\widehat{\BSD}}
\newcommand{\BSA}{\mathit{BSA}}
\newcommand{\BSAa}{\widehat{\BSA}}
\newcommand{\BSDA}{\mathit{BSDA}}
\newcommand{\BSDAa}{\widehat{\BSDA}}
\newcommand{\BSAA}{\mathit{BSAA}}
\newcommand{\BSAAa}{\widehat{\BSAA}}
\newcommand{\BSDD}{\mathit{BSDD}}
\newcommand{\BSDDa}{\widehat{\BSDD}}

\newcommand{\Bimod}[1]{{}_{#1}\ModCat_{#1}}

\newcommand{\cO}{\mathcal{O}}

\newcommand{\Tate}{\mathrm{Tate}}

\newcommand{\Tot}{\mathrm{Tot}}
\newcommand{\HFLa}{\widehat{\mathit{HFL}}}
\newcommand{\bAlg}{\overline{\Alg}}

\newcommand{\resp}{resp.\ }

\newcommand{\aprime}{a'}
\newcommand{\bprime}{b'}
\newcommand{\cprime}{c'}

\newcommand{\fix}{\mathrm{fix}}
\newcommand{\Borel}{\mathrm{Borel}}

\newcommand{\QH}{\mathit{QH}}

\begin{document}
\title[Double Covers]{Noncommutative Hodge-to-de Rham spectral sequence and the Heegaard Floer homology of double covers}

\author{Robert Lipshitz}
\thanks{RL was partially supported by NSF Grant number DMS-0905796 and a Sloan
  Research Fellowship.}
\address{Department of Mathematics, Columbia University\\
  New York, NY 10027}
\email{lipshitz@math.columbia.edu}

\author{David Treumann}
\address{Department of Mathematics, Boston College\\
  Chestnut Hill, MA 02467}
\email{treumann@bc.edu}

\begin{abstract}
  Let $A$ be a \dg algebra over $\Field$ and let $M$ be a \dg
  $A$-bimodule. We show that under certain technical hypotheses on
  $A$, a noncommutative analog of the Hodge-to-de Rham spectral sequence starts at the
  Hochschild homology of the derived tensor product  $M\DTP_A M$ and converges to the
  Hochschild homology of $M$. We apply this result to bordered Heegaard
  Floer theory, giving spectral sequences associated to Heegaard Floer
  homology groups of certain branched and unbranched double covers.
\end{abstract}

\maketitle

\tableofcontents

\section{Introduction}
This paper is inspired by a theorem of Hendricks and a question of
Lidman. In turn, they are:
\begin{citethm}\label{thm:hendricks}\cite[Theorem 1.1]{Hendricks12:rank-inequal}
  Let $K\subset S^3$ be a knot and $\pi\co \Sigma(K)\to
  S^3$ the double cover of $S^3$ branched along $K$. For $n$
  sufficiently large there is a spectral sequence with $E^1$-page
  given by the knot Floer homology group $\HFKa(\Sigma(K),\pi^{-1}(K))\otimes H_*(T^n)$ converging
  to $\HFKa(S^3,K)\otimes H_*(T^n)$.
\end{citethm}
(Here, $\HFKa(Y,K)$ denotes the knot Floer homology group of
$(Y,K)$~\cite{OS04:Knots,Rasmussen03:Knots} with coefficients in $\Field$, and $H_*(T^n)$ denotes
the singular homology of the $n$-torus.)

Hendricks deduces Theorem~\ref{thm:hendricks} from Seidel-Smith's
localization theorem for Lagrangian intersection Floer
homology~\cite{SeidelSmith10:localization}. In particular, the proof
is basically analytic. Lidman asked:
\begin{question}\label{question:lidman}
  (Lidman) Is it possible to recover Theorem~\ref{thm:hendricks} from
  cut-and-paste arguments?
\end{question}

In this paper we give a partial affirmative answer to
Question~\ref{question:lidman}; moreover, our
techniques can be used in situations where the hypotheses of
Seidel-Smith's theorem fail.
The idea is as follows. Bordered Floer
homology allows one to interpret the knot
Floer homology of $K$ as the Hochschild homology of a
bimodule~\cite[Theorem 14]{LOT2}. In characteristic 2 we show that
there is a spectral sequence which under certain technical hypotheses (see Theorem~\ref{thm:hoch-local}) has the form
\begin{equation}
  \label{eq:localization-intro}
  \HH_*(M\DTP_A M)\Rightarrow \HH_*(M),.
\end{equation}
where $M\DTP_A M$ denotes the derived tensor product (over $A$) of $M$ with itself. %
If the technical hypotheses are satisfied for the algebras in bordered
Floer theory, the spectral sequence~\eqref{eq:localization-intro}
gives another proof of Theorem~\ref{thm:hendricks}, as well as many
generalizations.

The technical hypotheses needed for~\eqref{eq:localization-intro} in
the case of bordered Floer homology boil down to a fairly concrete,
combinatorial problem. We have not been able to solve this problem in
general, but do give two partial results along these lines. Thus, we
obtain localization results for Heegaard Floer and knot Floer homology
groups, different from but overlapping with
Theorem~\ref{thm:hendricks}:

\begin{theorem}\label{thm:br-g2}
  Let $Y^3$ be a closed $3$-manifold, $K\subset Y$ a nullhomologous
  knot and $\spinc$ a torsion $\SpinC$-structure on $Y\setminus
  K$. Suppose that $K$ has a genus $2$ Seifert surface $F$. Then for
  each Alexander grading $i$ there is a spectral sequence
  \[
  \HFKa(\Sigma(K),\pi^{-1}(K);\pi^*\spinct,i)\Rightarrow
  \HFKa(Y,K;\spinct, i).
  \]
\end{theorem}
(This is proved in Section~\ref{sec:br-d-cov}. A simplified statement
in the special case of knots in $S^3$ is given as
Corollary~\ref{cor:knot-s-seq-g2-S3}.)

\begin{theorem}\label{thm:br-extreme}
  Let $Y^3$ be a closed $3$-manifold, $K\subset Y$ a nullhomologous
  knot and $\spinc$ a torsion $\SpinC$-structure on $Y\setminus
  K$. Let $F$ be a Seifert surface for $K$, of some genus $k$. Then
  there is a spectral sequence
  \[
  \HFKa(\Sigma(K),\pi^{-1}(K);\pi^*\spinct,k-1)\Rightarrow
  \HFKa(Y,K;\spinct, k-1).
  \]
\end{theorem}
(Again, this is proved in Section~\ref{sec:br-d-cov}.)

Our techniques also apply to certain unbranched double
covers. Specifically, let $Y$ be a closed $3$-manifold and
$\pi\co\tilde{Y}\to Y$ a $\ZZ/2$-cover. Viewing $\pi$ as an element of
$H^1(Y;\Field)$, assume $\pi$ is in the image of $H^1(Y;\ZZ)$. In this
case we say that $\pi$ is \emph{induced by a $\ZZ$-cover}
(Definition~\ref{def:from-Z-cover}). 

\begin{theorem}\label{thm:honest-dcov}
  Let $Y$ be a closed $3$-manifold and $\pi\co\tilde{Y}\to Y$ a
  $\ZZ/2$-cover which is induced by a $\ZZ$-cover. Let
  $\spinc\in\SpinC(Y)$ be a torsion $\SpinC$-structure.
  Then there is a spectral sequence
  \[
  \HFa(\tilde{Y};\pi^*\spinc)\otimes H_*(S^1)\Rightarrow \HFa(Y;\spinc).
  \]
\end{theorem}
(This is proved in Section~\ref{subsec:double-cover}.)

Theorems~\ref{thm:br-g2} and~\ref{thm:br-extreme} for knots in $S^3$
are, modulo the $H_*(T^n)$ factors and decomposition according to
Alexander gradings, special cases of Hendricks's
Theorem~\ref{thm:hendricks}. Theorems~\ref{thm:br-g2}
and~\ref{thm:br-extreme} for knots in other $3$-manifolds, as well as
Theorem~\ref{thm:honest-dcov}, seem not to be accessible via
Hendricks's techniques. Specifically, a Chern class computation shows
that the stable normal triviality condition required by Seidel-Smith always
fails in these cases; see~\cite[Remark 7.1]{Hendricks12:rank-inequal}.

The spectral sequence~\eqref{eq:localization-intro} is closely related to the
noncommutative
Hodge-to-de Rham spectral sequence (i.e.\ the Hochschild-to-cyclic spectral
sequence). For instance, when $A$ is Calabi-Yau, we show that the technical
condition on $A$ giving~\eqref{eq:localization-intro} is satisfied whenever
the Hodge-to-de Rham spectral sequence degenerates.
Also recall that the Hodge-to-de Rham spectral sequence comes by
analyzing an action of $\mathrm{U}(1)$ on the Hochschild chain complex
of $A$.  The full rotation group does not act on the Hochschild chain
complex of a bimodule, but the subgroup $\ZZ/2 \subset \mathrm{U}(1)$
does act on the Hochschild chain complex of the tensor square of a
bimodule.  The spectral sequence~\eqref{eq:localization-intro} comes
by analyzing this action.

\begin{remark}
\label{rem:augustus}
There is another resemblance between the algebra in this paper and the noncommutative Hodge-to-de Rham spectral sequence, about which we understand less. Whether or not our technical condition (``$\pi$-formality'') holds, we construct a spectral sequence starting at $\HH_*(M\DTP_A M)$, but  we cannot always identify its $E_{\infty}$-page.  When $M$ is $\pi$-formal, the identification $\HH_*(M) \stackrel{\sim}{\to} E_\infty$ is a kind of squaring map, but this map is not well-defined at the level of Hochschild chains.  There is (as has been pointed out to us independently by Yan Soibelman, Tyler Lawson, and the referee), a similar phenomenon at the heart of Kaledin's work \cite{Kaledin1} on the degeneration of the Hodge-to-de Rham spectral sequence: a squaring or more general Frobenius map defined on Hochschild homology of algebras (with values in a form of cyclic homology) that is not induced by a map of chain complexes.  An explanation in terms of stable homotopy is given in \cite{Kaledin2} --- it would be interesting to see if this explanation applies in our setup as well.
\end{remark}

Beyond bordered Floer homology, there
are a number of other cases in which one could try to apply the
spectral sequence~\eqref{eq:localization-intro} (i.e.,
Theorem~\ref{thm:hoch-local}). One obvious class of examples is
provided by Khovanov and Khovanov-Rozansky knot homologies. Another
comes from Fukaya categories. 
Let $(M,\omega)$ be a
symplectic manifold and $\phi\co M\to M$ a symplectomorphism. Then
$\phi$ induces an automorphism $\phi_*$ of the Fukaya category
$\Fuk(M)$ of $M$. According to the philosophy
of~\cite{Kontsevich95:ICM,Seidel09:SH-HH}, if $M$ contains enough
Lagrangians then $\Fuk(M)$ controls the Floer theory of $M$. A special
case of this is the following well-known folk conjecture:
\begin{citeconj}\label{conj:fp-Floer} Let $(M,\omega)$ be a symplectic
  manifold for which the Fukaya category $\Fuk(M)$ of $M$ and the
  quantum cohomology $\QH^*(M)$ of $M$ are defined over $\Field$.
  Suppose further that the natural map $\HH_*(\Fuk(M))\to \QH^*(M)$ is
  an isomorphism. Let $\phi\co M\to M$ be a symplectomorphism with
  fixed-point Floer homology $\HF(\phi)$. Then
  \begin{equation}\label{eq:seidel-conj}
    \HF(\phi)\cong \HH_*\bigl(\phi_*\co \Fuk(M)\to\Fuk(M)\bigr).
  \end{equation}
\end{citeconj}

Thus, for $M$ as in the statement of Conjecture~\ref{conj:fp-Floer},
when $\Fuk(M)$ satisfies (appropriate analogues of) the
technical hypotheses of Theorem~\ref{thm:hoch-local}, the spectral
sequence~\eqref{eq:localization-intro} implies that
\begin{equation}\label{eq:fixed-pt-floer}
\dim\HF(\phi^2)\geq \dim\HF(\phi).
\end{equation}
This inequality has nontrivial consequences. For example, for $\tau$
the hyperelliptic involution of a genus $g$ surface, it is easy to see
that $\HF(\tau)$ has dimension $2g+2$: the $2g+2$ fixed points of
$\tau$ lie in different Nielsen
classes. Formula~\eqref{eq:fixed-pt-floer} then implies that any
(non-degenerate) map Hamiltonian-isotopic to $\tau^2=\Id$ has at least
$2g+2$ fixed points, a statement which does not hold for arbitrary
smooth maps in the isotopy class. (Of course, this result also follows
from the Arnold conjecture.)

In the special case of area-preserving diffeomorphisms of a surface
with boundary $S^1$, it
should be possible to combine Theorem~\ref{thm:br-extreme} with the
isomorphisms between Heegaard Floer homology, embedded contact
homology, Seiberg-Witten Floer homology and periodic Floer
homology~\cite{Taubes10:SW-ECH-I,Taubes10:SW-ECH-II,Taubes10:SW-ECH-III,Taubes10:SW-ECH-IV,Taubes10:SW-ECH-V,LeeTaubes12:PFH-SW,KutluhanLeeTaubes:HFHMI,KutluhanLeeTaubes:HFHMII,KutluhanLeeTaubes:HFHMIII,KutluhanLeeTaubes:HFHMIV,KutluhanLeeTaubes:HFHMV,ColinGhigginiHonda11:HF-ECH-1,ColinGhigginiHonda11:HF-ECH-2,ColinGhigginiHonda11:HF-ECH-3}
to obtain the inequality~\eqref{eq:fixed-pt-floer} without using
Conjecture~\ref{conj:fp-Floer}.

This paper is organized as follows. Section~\ref{sec:review} gives a
brief review of $\ZZ/2$-localization for singular homology; this is
not needed for what follows, but should help elucidate the structure
of later arguments. Section~\ref{sec:hoch-local} is the algebraic part of the paper. We
start with a review of Hochschild homology
(Section~\ref{subsec:hoch-background}) and a short review of spectral
sequences associated to bicomplexes (Section~\ref{sec:ssfb}), partly
to fix notation. We then explain the basic algebraic condition, which
we call \emph{$\pi$-formality}, under which the spectral
sequence~\eqref{eq:localization-intro} holds
(Section~\ref{subsec:HHTate}). We then discuss when this condition
holds for all $A$-bimodules; this is \emph{$\pi$-formality of $A$}
(Section~\ref{sec:natural-formal}). For Theorems~\ref{thm:br-g2} and~\ref{thm:br-extreme}, this is all the algebra we need. For Theorem~\ref{thm:honest-dcov} we need one more notion, that of \emph{neutral bimodules}, bimodules on which the Serre functor acts trivially in a certain sense
(Section~\ref{sec:neutral}). (If $A$ is Calabi-Yau then every bimodule is neutral.) The last two subsections of Section~\ref{sec:hoch-local} do not (yet) have topological applications, but are included to help set $\pi$-formality in a broader context. Specifically, in Section~\ref{subsec:imapf} we discuss the case that $A$ admits an integral lift; in this case, $\pi$-formality is (in some sense) easier to verify. In Section~\ref{subsec:is-hodge-to-derham} we show that if $A$ is Calabi-Yau then
the condition of $\pi$-formality follows from collapse of the
Hodge-to-de Rham spectral sequence.

Section~\ref{sec:HF-applications} is devoted to applications of the
algebraic results to Heegaard
Floer homology. It starts by collecting background on bordered and
bordered-sutured Heegaard Floer homology
(Section~\ref{sec:bordered-background}); there, we also observe
homological smoothness for the relevant algebras. We discuss
$\pi$-formality of the bordered and bordered-sutured algebras
(Section~\ref{sec:loc-bord-cobar}). While $\pi$-formality in general
remains a conjecture, we verify this conjecture in several interesting
cases. The first application is to branched double covers of links,
giving Theorems~\ref{thm:br-g2} and~\ref{thm:br-extreme}
(Section~\ref{sec:br-d-cov}). We then discuss a particular
bordered-sutured $3$-manifold, the so-called \emph{tube-cutting piece}
(Section~\ref{sec:TC}) and, using this manifold, obtain a localization
result for ordinary double covers, Theorem~\ref{thm:honest-dcov}
(Section~\ref{subsec:double-cover}).

\subsection*{Acknowledgments}
We thank Mohammed Abouzaid, Tyler Lawson, Dan Lee, Ciprian Manolescu,
Junecue Suh and Yan Soibelman for helpful discussions. We especially
thank Tye Lidman for suggesting that bordered Floer homology might be
used to study covering spaces and for many corrections to a draft of
this paper, and Kristen Hendricks for her work inspiring these results
and for many helpful conversations. Finally, we thank the referee for
a careful reading and many helpful and interesting comments.

The ideas in Section~\ref{sec:TC} arose in discussions of
Peter~Ozsv\'ath, Dylan~Thurston and the first author, and were observed
independently by Rumen~Zarev.

\section{Review of \texorpdfstring{$\ZZ/2$}{Z/2}-localization for
  singular homology}\label{sec:review}
To ease into the algebra, we start by reviewing a particular
perspective on the localization theorem for $\ZZ/2$-equivariant singular
homology.

Consider a topological space $X$ with a $\ZZ/2$-action $\tau\co X\to
X$. The (Borel) \emph{equivariant cohomology} of $X$ is defined to be the
singular cohomology
\begin{equation}\label{eq:borel}
H^*_{\ZZ/2}(X;\ZZ)\coloneqq H^*(X\times_{\ZZ/2}E\ZZ/2;\ZZ),
\end{equation}
where $E\ZZ/2$ is a contractible space with a free
$\ZZ/2$-action (e.g., $E\ZZ/2=S^\infty$).

Equivalently, the $\ZZ/2$-action on $X$ induces a $\ZZ/2$-action on
the singular chains $C_*(X)$, i.e., makes $C_*(X)$ into a chain
complex over the group ring $\ZZ[\ZZ/2]$. So, we could define 
\begin{equation}
  \label{eq:ext-borel}
  H^*_{\ZZ/2}(X;\ZZ)\coloneqq \Ext_{\ZZ[\ZZ/2]}(C_*(X),\ZZ),
\end{equation}
where $\ZZ$ is given the trivial $\ZZ/2$-action. Since $C_*(X\times
E\ZZ/2)$ is a free resolution of $C_*(X)$ as a $\ZZ[\ZZ/2]$-module,
Equations~\eqref{eq:borel} and~\eqref{eq:ext-borel} are
equivalent. One advantage of Equation~\eqref{eq:ext-borel} is that it
allows one to define an equivariant homology for any chain complex over
$\ZZ[\ZZ/2]$. Another advantage is that it allows one to use other models for
$C_*(X)$, like the cellular chain complex for $X$ (if $X$ was a CW
complex and the $\ZZ/2$-action was cellular).

A particularly nice projective resolution of $\ZZ$ as a $\ZZ[\ZZ/2]$-module is
given by 
\[
0 \longleftarrow\ZZ[\ZZ/2]\stackrel{1-\tau}{\longleftarrow}\ZZ[\ZZ/2]\stackrel{1+\tau}{\longleftarrow}\ZZ[\ZZ/2]\stackrel{1-\tau}{\longleftarrow}\ZZ[\ZZ/2]\stackrel{1+\tau}{\longleftarrow}\cdots.
\]
(This resolution comes from thinking of the cellular chain complex for
the usual $\ZZ/2$-equivariant cell structure on $S^\infty$, say.)  
Tensoring over $\ZZ$ with $C_*(X)$ gives a projective resolution of $C_*(X)$ over $\ZZ[\ZZ/2]$
\begin{equation}\label{eq:borel-proj-res}
0 \longleftarrow C_*(X;\ZZ)\otimes \ZZ[\ZZ/2]\stackrel{1\otimes 1-1\otimes \tau}{\longleftarrow}C_*(X;\ZZ)\otimes\ZZ[\ZZ/2]\stackrel{1\otimes 1+1\otimes \tau}{\longleftarrow}C_*(X;\ZZ)\otimes\ZZ[\ZZ/2]\stackrel{1\otimes 1-1\otimes \tau}{\longleftarrow}\cdots
\end{equation}
where $\ZZ/2$ acts diagonally on each term.
So,
$H^*_{\ZZ/2}(X;\ZZ)$ is the homology of the total complex associated to
the bicomplex
\begin{equation}\label{eq:canonical-borel-cx}
C^*_{\Borel}(X;\ZZ)\coloneqq\bigl(0\longrightarrow C^*(X;\ZZ)\stackrel{1-\tau^*}{\longrightarrow} C^*(X;\ZZ)\stackrel{1+\tau^*}{\longrightarrow} C^*(X;\ZZ)\stackrel{1-\tau^*}{\longrightarrow} C^*(X;\ZZ)\stackrel{1+\tau^*}{\longrightarrow} \cdots\bigr)
\end{equation}
obtained from Formula~\eqref{eq:borel-proj-res} by taking $\Hom$ over
$\ZZ[\ZZ/2]$ to $\ZZ$.

The projection map $X\times_{\ZZ/2}E\ZZ/2\to
(E\ZZ/2)/(\ZZ/2)\eqqcolon B\ZZ/2\simeq \RR P^\infty$ endows
$H^*_{\ZZ/2}(X;\ZZ)$ with an action of $H^*(\RR P^\infty;\ZZ)$. Let $\theta\in
H^2(\RR P^\infty)\cong \ZZ/2$ be a generator. Multiplication by
$\theta$ annihilates $p^n$ torsion for any $p\neq 2$, so it is natural
to consider equivariant cohomology with $\Field$-coefficients. Over
$\Field$, $H^*(\RR P^\infty;\Field)\cong \Field[\eta]$, where $\eta\in
H^1(\RR P^\infty;\Field)$, and the localization
theorem states that under appropriate hypotheses,
\begin{equation}\label{eq:top-localization}
\eta^{-1}H^*_{\ZZ/2}(X;\Field)\coloneqq H^*_{\ZZ/2}(X;\Field)\otimes_{H^*(B\ZZ/2;\Field)}
\Field[\eta,\eta^{-1}] \cong H^*(X^{\fix};\Field)\otimes_\Field \Field[\eta,\eta^{-1}],
\end{equation}
where $X^\fix$ denotes the fixed set of $\tau$.

Inverting $\eta$ before taking cohomology allows us to give a
chain-level statement of the localization theorem.
That is, consider the \emph{Tate complex} of $(X,\tau)$
\[
C^*_{\Tate}(X;\Field)\coloneqq\bigl(\cdots
\stackrel{1+\tau}{\longrightarrow}C^*(X;\Field)\stackrel{1+\tau}{\longrightarrow}C^*(X;\Field)\stackrel{1+\tau}{\longrightarrow}C^*(X;\Field)\stackrel{1+\tau}{\longrightarrow}C^*(X;\Field)\stackrel{1+\tau}{\longrightarrow}\cdots\bigr),
\]
a periodic analogue of $C^*_{\Borel}$. The localization theorem is
then the statement that the Tate equivariant cohomology satisfies
$H^*_{\Tate}(X;\Field)\coloneqq h_*(C^*_{\Tate}(X;\Field))\cong H^*(X^{\fix};\Field)\otimes_\Field \Field[\eta,\eta^{-1}]$.

In the paper, we will actually work with $\ZZ/2$-equivariant homology,
i.e.,
\[
H_*^{\ZZ/2}(X;\Field)=H_*(X\times_{\ZZ/2}E\ZZ/2;\Field)=\Tor_{\Field[\ZZ/2]}(C_*(X),\Field).
\]
For homology, the localization theorem can be stated as follows:
\begin{citethm}\label{citethm:classical-localization}
  Let $X$ be a finite-dimensional CW complex, and let $\tau\co X\to X$
  be an involution with fixed set $X^{\fix}$. Consider the Tate
  complex
  \[
  C_*^{\Tate}(X;\Field)=\bigl(\cdots
  \stackrel{1+\tau}{\longleftarrow}C_*(X;\Field)\stackrel{1+\tau}{\longleftarrow}C_*(X;\Field)\stackrel{1+\tau}{\longleftarrow}C_*(X;\Field)\stackrel{1+\tau}{\longleftarrow}C_*(X;\Field)\stackrel{1+\tau}{\longleftarrow}\cdots\bigr).
  \]  
  Then the Tate
  equivariant homology $H_*^{\Tate}(X;\Field)\coloneqq h_*(C_*^{\Tate}(X;\Field))$ is
  isomorphic to the tensor product $H_*(X^{\fix};\Field)\otimes_\Field
  \Field[\eta,\eta^{-1}]$.
\end{citethm}
\begin{proof}
  There are two obvious spectral sequences associated to the bicomplex
  $C_*^{\Tate}(X)$, depending on whether we take homology first with
  respect to the differential on $C_*(X;\Field)$ or first with respect to the
  $1+\tau$ differentials. Call these two spectral sequences
  $\vhE^r_{p,q}$ and $\hvE^r_{p,q}$, respectively. (For some details about our
conventions on spectral sequences, see Section~\ref{sec:ssfb}.)
  Consider first page of the $\hvE$ spectral sequence. The kernel of $1+\tau$
  has two kinds of generators:
  \begin{itemize}
  \item Generators $\sigma\co \Delta^n\to X^\fix$ contained in the
    fixed set of $\tau$. (These are exactly the generators with $\sigma=\tau_*\sigma$.)
  \item Sums $\sigma+\tau\circ\sigma$ where the image of $\sigma$ is not contained in $X^\fix$.
  \end{itemize}
  The image of $1+\tau$ is exactly the second set of generators. Thus,
  the $E^1$-page of the spectral sequence is identified with
  $C_*(X^\fix;\Field)$. By definition, the differential on the $\hvE^1$-page is
  exactly the simplicial cochain differential on
  $C_*(X^\fix;\Field)$. Moreover, the spectral sequence collapses at $E^2$,
  since any generator in the $\hvE^2$-page has a representative which is
  a cycle for both the differential on $C_*(X;\Field)$ and the differential
  $1+\tau$ (cf.~Remark~\ref{rem:hv}).

  Thus, $\hvE^\infty$ is $H_*(X^{\fix};\Field)\otimes_\Field
  \Field[\eta,\eta^{-1}]$. The hypothesis that $X$ is a
  finite-dimensional CW complex provides enough boundedness to ensure
  that this limit is, in fact, the homology of the original chain
  complex $C_*^{\Tate}(X;\Field)$.
\end{proof}

\begin{corollary}\label{cor:classical-loc-seq}
  There is a spectral sequence whose $E^1$-page is
  $H_*(X;\Field)\otimes \Field[\eta,\eta^{-1}]$ and whose
  $E^\infty$-page is $H_*(X^\fix;\Field)\otimes
  \Field[\eta,\eta^{-1}]$.
\end{corollary}
\begin{proof}
  This follows by considering the $\vhE$ spectral sequence. It is
  immediate from the definition that $\vhE^1$ is
  $H_*(X;\Field)\otimes \Field[\eta,\eta^{-1}]$.  The fact that
  $X$ is a finite-dimensional CW complex ensures that this spectral
  sequence converges to the homology of $C_*^{\Tate}(X;\Field)$ which, by
  Theorem~\ref{citethm:classical-localization}, is exactly
  $H_*(X^\fix;\Field)\otimes \Field[\eta,\eta^{-1}]$.
\end{proof}
The corollary implies the classical Smith inequality: $\dim
H_*(X^{\fix};\Field)\leq \dim H_*(X;\Field)$. 

In proving Theorem~\ref{citethm:classical-localization} and
Corollary~\ref{cor:classical-loc-seq} there were two key points:
\begin{enumerate}
\item\label{item:collapse} The $\hvE$ spectral sequences associated to
  the Tate bicomplex collapses at the $E^2$-page, allowing us to
  identify the limit. (By contrast, the $\vhE$ spectral sequence,
  appearing in Corollary~\ref{cor:classical-loc-seq}, can be
  arbitrarily complicated.)
\item\label{item:boundedness} A boundedness condition---here, that $X$
  is a finite-dimensional CW complex---allows us to identify the
  limits of the $\hvE$ and $\vhE$ spectral sequences with the homology
  of the Tate complex itself.
\end{enumerate}
In the discussion of Hochschild homology below, the boundedness
property~(\ref{item:boundedness}) will be replaced by the condition of
``homological smoothness'' (Definition~\ref{def:homol-smooth}). We
will be interested in conditions under which the spectral sequence
$\hvE$ collapses (at the $E^3$- rather than $E^2$-page, it turns out);
we call this collapse ``$\pi$-formality''
(Definition~\ref{def:pi-formal-bimod}). Like
Corollary~\ref{cor:classical-loc-seq},
Theorems~\ref{thm:br-g2},~\ref{thm:br-extreme},~\ref{thm:honest-dcov}
and their algebraic archetype, Theorem~\ref{thm:hoch-local}, will then
come from the other ($\vhE$) spectral sequence; and this spectral
sequence can in principle be arbitrarily complicated.

\section{\texorpdfstring{$\ZZ/2$}{Z/2}-Localization in Hochschild homology}
\label{sec:hoch-local}

Let $A$ be a dg algebra over $\Field$, let $M$ be a dg bimodule over
$A$, and let $\HH_*(A,M)$ denote the Hochschild homology of $M$.  In
this section, we construct a natural operation $d^4:\HH_k(A,M) \to
\HH_{k-2}(A,M)$, along with higher order operations $d^{2i}:\HH_k(A,M)
\dashrightarrow \HH_{k - i}(A,M)$ for $i > 2$, and investigate what we
call \emph{$\pi$-formality} (Definition~\ref{def:pi-formal-bimod}), the vanishing of all of these operations.  

We say that a bimodule $M$ is $\pi$-formal if $d^{2i}$ vanishes on $\HH_*(A,M)$ for every $i$.  We say that a dg algebra $A$ is $\pi$-formal if every $(A,A)$-bimodule is $\pi$-formal.  We will give several sufficient conditions for $\pi$-formality.  Our main result is the identification of the $E_\infty$-page of a ``localization'' spectral sequence for $\pi$-formal bimodules.

\begin{theorem}
\label{thm:hoch-local}
Let $A$ be a dg algebra over $\Field$, let $M$ be an $(A,A)$ dg bimodule, and let $M \DTP M$ denote the derived tensor product, over $A$, of $M$ with itself.  Suppose that:
\begin{enumerate}[label =(A-\arabic*),ref=(A-\arabic*)]
\item $A$ has finite dimensional homology over $\Field$, and is perfect as an $(A,A)$-bimodule.  In the language of \cite[Section 8]{KontsevichSoibelman06:KoSo06}, $A$ is \emph{homologically smooth and proper}.
\item $M$ is bounded, i.e., supported in finitely-many gradings.
\item $M$ is $\pi$-formal.
\end{enumerate}
Then there is a spectral sequence starting at $\HH_*(A,M \DTP M)$ and
converging to $\HH_*(A,M)$.  (Here, $\DTP$ denotes the derived
tensor product over $A$.)

More precisely, there is a spectral sequence $\vhE^r_{p,q}$ for which the following hold:
\begin{enumerate}
\item\label{item:hoch-local-E1} For all $p$ and $q$,
\[
\vhE^1_{p,q} = \HH_q(A,M \DTP M).
\]
\item\label{item:hoch-local-Einfty} There is an increasing filtration $\hF_i$ of $V \cong \bigoplus_j \HH_j(A,M)$ such that
\[
\vhE^\infty_{p,q} = \hF_{-q} V /\hF_{-q-1}V.
\]
\end{enumerate}
In particular, there is a rank inequality
\[
\sum_q \dim_{\Field}(\HH_q(A, M \DTP M)) \geq \sum_q \dim_\Field(\HH_q(A,M)).
\]
\end{theorem}

\subsection{Background on dg algebras and Hochschild homology}\label{subsec:hoch-background}

By a \emph{chain complex} we will mean a complex with a differential of degree $-1$.  Write $h_i(C)$ for the $i\th$ homology of $C$.  We denote the shift of $C$ by $\Sigma C$, i.e.  $(\Sigma C)_k = C_{k-1}$.

We will usually work over $\Field$ or $\ZZ$.  Let $D(\Field)$ (\resp $D(\ZZ)$) denote the derived category of $\Field$-vector spaces (\resp abelian groups).

A \emph{dg algebra} is a chain complex $A = (A_*,\partial)$ of $\Field$- or $\ZZ$-modules equipped with an associative multiplication satisfying:
\begin{itemize}
\item $a \cdot b \in A_{i+j}$ whenever $a \in A_i$ and $b \in A_j$
\item $\partial(a \cdot b) = \partial(a) \cdot b + (-1)^{|a|} a \cdot \partial(b)$
\end{itemize}
When working over $\ZZ$, we will always assume $A$ is free as a $\ZZ$-module.  If $A$ is a dg algebra, an \emph{$(A,A)$-bimodule} is a chain complex $M = (M_*,\bdy)$
equipped with a graded $(A_*,A_*)$-bimodule structure on $M_*$ and such that 
$\bdy(a \cdot m \cdot b) = \bdy(a) \cdot m \cdot b + (-1)^{|a|} a
\cdot \bdy(m) \cdot b + (-1)^{|a||m|} a \cdot m \cdot \bdy(b)$.  
Let $D(\Bimod{A})$ denote the derived category of $(A,A)$ dg bimodules, obtained by inverting quasi-isomorphisms in the homotopy category of $(A,A)$-bimodules.

Unless otherwise noted, $\otimes$ will denote tensor product over the
ground ring $\Field$ or $\ZZ$.

\subsubsection{Resolutions and perfect bimodules}
\label{subsubsec:rapb}
For $A$ a dg algebra over $\Field$ or $\ZZ$, the total complex of the bicomplex $A \otimes A$ is equipped with an $(A,A)$-bimodule structure by setting $a \cdot (b \otimes c) \cdot d = (ab) \otimes (cd)$.  We denote this bimodule by $A^e$ and call it the ``free $(A,A)$-bimodule of rank $1$ in degree zero.''  In general we say that a dg bimodule is free if it is of the form $\bigoplus_{i \in I} \Sigma^{s_i} A^e$, and that it has finite rank if $I$ is finite.

A \emph{cell bimodule} is any bimodule $C$ that admits a filtration $C_1 \subset C_2 \subset \cdots$ such that $C_i/C_{i-1}$ is isomorphic (not just quasi-isomorphic) to a free bimodule.  We say $C$ is a \emph{finite cell bimodule} if the filtration can be chosen finite with each subquotient free of finite rank.

A \emph{cell retract} (\resp \emph{finite cell retract}) is subcomplex $R$ of a cell bimodule (\resp finite cell bimodule) $C$ such that the inclusion $R \to C$ admits an $(A,A)$-bimodule retract $r: C \to R$.  A \emph{resolution} of bimodule $M$ is a quasi-isomorphism $R \to M$ where $R$ is a cell retract.  An object of $\Bimod{A}$ is called \emph{perfect} if it admits a resolution by a finite cell retract.

\begin{definition}
\label{def:homol-smooth}
Let $A$ be a dg algebra over $\Field$ (\resp over $\ZZ$)
\begin{itemize}
\item $A$ is called \emph{homologically proper} if the homology $\bigoplus_{i \in \ZZ} h_i(A)$ is finite dimensional (\resp finitely generated).
\item $A$ is called \emph{homologically smooth} if it is perfect as an $(A,A)$-bimodule.
\end{itemize}
\end{definition}

\subsubsection{Tensor product}
If $M$ and $N$ are $(A,A)$-bimodules, we may define a naive tensor product bimodule $M \boxtimes_A N$ by endowing the graded tensor product $M_* \otimes_{A_*} N_*$ with the differential $\bdy(m \otimes n) = \bdy(m) \otimes n + (-1)^{|m|} m \otimes \bdy(n)$.  We may similarly define a naive tensor product $M_1 \boxtimes_A \cdots \boxtimes_A M_k$ of any number of dg bimodules.

The naive tensor product does not respect quasi-isomorphisms.  We define a corrected or derived version $\DTP$ of the tensor product by fixing a resolution $R \to A$ of the diagonal bimodule $A$ and setting
\[
M \DTP N := M \boxtimes_A R \boxtimes_A N
\]
This induces a bifunctor $\DTP:D(\Bimod{A}) \times D(\Bimod{A}) \to D(\Bimod{A})$.

\subsubsection{Hochschild homology}\label{sec:hoch-background}

\begin{definition}\label{def:Hochschild}
Let $A$ be a dg algebra over $\Field$ (\resp $\ZZ$) and let $R \to A$ be a resolution of $A$ as an $(A,A)$-bimodule.  Let $M$ be an $(A,A)$-bimodule.  The \emph{Hochschild chain complex} of $M$ is the quotient of the total complex of
$R_* \otimes_\Field M_*$ (\resp $R_* \otimes_{\ZZ} M_*$) by the equivalence relation generated by
\[
ra \otimes m \sim r \otimes am \qquad ar \otimes m \sim (-1)^{|a|(|r|+|m|)}r \otimes ma
\]
and with differential given by
\[
\bdy(r \otimes m) = \bdy(r) \otimes m + (-1)^{|r|} r \otimes \bdy(m)
\]
Let $\HC(M) = \HC(A,M)$ denote the Hochschild chain complex of $M$, and set $\HH_i(M) = h_i(\HC(M))$; $\HH_i(M)$ is the $i\th$ \emph{Hochschild homology} group of $M$. (More abstractly, $\HC(M)$ is the derived tensor product of $A$ and $M$ in the category of bimodules---or $A\otimes A^\op$-modules---and $\HH(M)=\Tor_{A\otimes A^\op}(A,M)$.)
\end{definition}

The assignment $M \mapsto \HC(M)$ is functorial, and carries quasi-isomorphisms to quasi-isomorphisms, thus $\HH_i(M)$ is a functor from $D(\Bimod{A})$ to $D(\Field)$ or $D(\ZZ)$.  When $A$ is smooth and proper, this functor is representable (see for instance \cite[Remark~8.2.4]{KontsevichSoibelman06:KoSo06})

\begin{proposition}
\label{prop:cobar}
Suppose $A$ is homologically smooth and proper.  Then there is an $(A,A)$ \dg bimodule $A^!$, unique up to quasi-isomorphism, and a natural isomorphism
\[
\Hom(\Sigma^k A^!, M) \cong \HH_k(M)
\]
where the $\Hom$ on the left-hand side indicates the group of homomorphisms in the derived category $D(\Bimod{A})$.
\end{proposition}

Because of this, any natural transformation $\HH_k(M) \to
\HH_{k+r}(M)$ comes from a map $\Hom(\Sigma^{k+r} A^!,\Sigma^{k} A^!)
\cong \HH_{r}(A^!)$.  
In \cite[Definition 8.1.6]{KontsevichSoibelman06:KoSo06} $A^!$ is
called the ``inverse dualizing bimodule.''  If $P$ is any complex of
projective $(A,A)$-bimodules resolving the diagonal bimodule $A$, then $A^!$ is quasi-isomorphic to $\Hom_{\Bimod{A}}(P,A^e)$.  Since $P$ can be taken to be the bar resolution of $A$, we will call $A^!$ the ``cobar bimodule'' for short.  A smaller Koszul resolution will be useful to us in our applications in Section \ref{sec:HF-applications}.

\subsection{Spectral sequences from bicomplexes}
\label{sec:ssfb}

For us, a bicomplex is either a bigraded free $\ZZ$-module or, more often, a bigraded $\Field$-vector space $C_{*,*}$, together with differentials, $d^h:C_{p,q} \to C_{p-1,q}$ and $d^v:C_{p,q} \to C_{p,q-1}$, such that $d^h \circ d^v + d^v \circ d^h=0$.

Write $\Tot(C)$ for the total complex of $C$, i.e.
\[
\Tot(C)_n = \bigoplus_{p+q = n} C_{p,q}
\]
with differential given by $d(x) = d^v(x) + d^h(x)$.

We will denote the two standard filtrations on a bicomplex by $\vF$ and $\hF$, namely
\[
\begin{array}{ccc}
(\vF_k C)_{p,q} & = & \bigg\{
\begin{array}{cl}
C_{p,q} & \text{if }q \leq k\\
0 & \text{otherwise}
\end{array}\\
(\hF_k C)_{p,q} & = & \bigg\{
\begin{array}{cl}
C_{p,q} & \text{if }p \leq k\\
0 & \text{otherwise}
\end{array}
\end{array}
\]
These filtrations induce spectral sequences which we will denote by $\hvE$ (attached to $\vF$) and $\vhE$ (attached to $\hF$).  By computing first the horizontal homology and then the vertical homology of the bicomplex, we obtain $\hvE^2$, and by computing the reverse we obtain $\vhE^2$.

\begin{remark}
\label{rem:hv}
We will compute differentials in these spectral sequences by the following standard device.  If $x \in C_{p,q}$ is an element that survives to $\vhE^r_{p,q}$, and $(x_1,\ldots,x_r)$ is a sequence of elements with $x = x_1$ and $d^h(x_i) = d^v(x_{i+1})$ for $i< r$, then $d^h(x_r)$ is a representative for $\vhd^r(x)$ in $\vhE^r_{p-r,q+r-1}$.  (We will call such a sequence a \emph{vh sequence}).  Similarly if $y\in C_{p,q}$ survives to $\hvE^r_{p,q}$ and $(x_1,\ldots,x_r)$ is a sequence of elements with $y = y_1$ and $d^v(y_i) = d^h(y_{i+1})$ for $i < r$ (an \emph{hv sequence}), then $d^v(y_r)$ is a representative for $\hvd^r(y)$.
\end{remark}

\begin{remark}\label{remark:vh-pictures}
Our grading conventions for $\hvE$ are transposed from the standard ones, that is we write $\hvE_{p,q}$ for what is more typically called $E_{q,p}$.  Here are $\hvE^0$,$\hvE^1$, and $\hvE^2$:

\begin{align*}
\xymatrix{
\bullet & \bullet \ar[l] & \bullet \ar[l] \\
\bullet & \bullet \ar[l] & \bullet \ar[l] \\
\bullet & \bullet \ar[l] & \bullet \ar[l] 
}
& &
\xymatrix{
\bullet \ar[d] & \bullet \ar[d] & \bullet \ar[d] \\
\bullet \ar[d] & \bullet \ar[d] & \bullet \ar[d]\\
\bullet & \bullet & \bullet 
}
& &
\xymatrix{
\bullet \ar[ddr] & \bullet \ar[ddr]  & \bullet  \\
\bullet & \bullet  & \bullet \\
\bullet & \bullet  & \bullet  
}
\end{align*}

Our grading conventions for $\vhE$ are standard.  Here is a diagram of the pages $\vhE^0$, $\vhE^1$, and $\vhE^2$:
\begin{align*}
\xymatrix{
\bullet \ar[d] & \bullet \ar[d] & \bullet \ar[d] \\
\bullet \ar[d] & \bullet \ar[d] & \bullet \ar[d]\\
\bullet & \bullet & \bullet 
}
& &
\xymatrix{
\bullet & \bullet \ar[l] & \bullet \ar[l] \\
\bullet & \bullet \ar[l] & \bullet \ar[l] \\
\bullet & \bullet \ar[l] & \bullet \ar[l] 
}
& &
\xymatrix{
\bullet & \bullet  & \bullet  \\
\bullet & \bullet  & \bullet \ar[llu] \\
\bullet & \bullet  & \bullet \ar[llu] 
}
\end{align*}
\end{remark}

Under suitable boundedness conditions, the final pages $\vhE^\infty$ and $\hvE^\infty$ are related to the homology of $\Tot(C)$.  Note that the homology of $\Tot(C)$ carries filtrations
\[
\begin{array}{c}
 \hF_p H_n(\Tot(C)) = \{z \in H_n(\Tot(C)) \mid z \text{ is represented by a cycle in $\bigoplus_{i \leq p} C_{i,n-i}$}\} \\
 \vF_p H_n(\Tot(C)) = \{z \in H_n(\Tot(C)) \mid z \text{ is represented by a cycle in $\bigoplus_{i \leq p} C_{n-i,i}$}\}
 \end{array}
 \]
 
 \begin{proposition}
\label{prop:demoted}
Suppose that, for each $n$, there are only finitely many $p$ such that $C_{p,n-p} \neq 0$.  Then 
\[
\begin{array}{c}
\vhE^\infty_{p,q} = \hF_p H_{p+q}(\Tot(C))/\hF_{p-1} H_{p+q}(\Tot(C))\\
\hvE^\infty_{p,q} = \vF_p H_{p+q}(\Tot(C))/\vF_{p-1} H_{p+q}(\Tot(C)).
\end{array}
\]
\end{proposition}

\begin{proof}
  This is standard; see, for instance,~\cite[Theorem
  3.2]{McCleary01:UsersGuide}.
\end{proof}

\subsection{The Hochschild-Tate bicomplex and the operations \texorpdfstring{$d^{2i}$}{d-2i}}
\label{subsec:HHTate}

We construct operations $d^{2i}$ on $\HH_*(M)$ by considering the
bimodule $M \DTP M$ and its Hochschild chains $\HC(M \DTP M)$.  In
this section we work over $\Field$.  The following proposition is key:

\begin{proposition}
\label{prop:Z2action}
The map $\tau:\HC(M \DTP M) \to \HC(M \DTP M)$ that sends $r \otimes (m \otimes r' \otimes m')$ to $r' \otimes (m' \otimes r \otimes m)$ is a map of chain complexes, and satisfies $\tau \circ \tau(x) = x$.  Moreover, if $A$ is homologically smooth and proper then we may choose an $\Field$-basis of $\HC(M)$ of the form $\{r_i \otimes m_i\}_{i \in I}$ such that $\{r_i \otimes (m_i \otimes r_j \otimes m_j)\}_{i,j \in I}$ is an $\Field$-basis for the chain complex $\HC(M \DTP M)$.
\end{proposition}

Note that $\tau$ is not induced by a bimodule homomorphism $M \DTP M \to M \DTP M$.

\begin{proof}
It is easy to see that the map $\tau$ commutes $\partial_{\HC(M \DTP M)}$.  Let us prove the second assertion.

Since $A = (A_*,\bdy_A)$ is homologically proper, we may assume that $A_*$ is finite-dimensional over $\Field$.  Since $A$ is homologically smooth, we may assume that $R_*$ is finite-dimensional and projective as an $(A_*,A_*)$-bimodule.  We will show that, if $A_*$ is any finite-dimensional algebra and $R_*$ is a finite-dimensional projective $(A_*,A_*)$-bimodule, then 
$R_* \otimes M_*/\sim$ has a basis $B = \{r_i \otimes m_i\}$ such that $\{r_i \otimes m_i \otimes r_j \otimes m_j\}$ is a basis for $R_* \otimes (M_* \otimes_{A_*} R_* \otimes_{A_*} M_*)/\sim$.

It suffices to prove the claim for indecomposable projective bimodules, i.e. we may assume $R_* = e A_* \otimes_\Field A_* f$ where $e$ and $f$ are principal idempotents in $A_*$.  In that case it is easy to verify the following:
\begin{enumerate}
\item $(R_* \otimes M_*)/\sim$ is naturally identified with $e M_* f$
\item $R_* \otimes (M_* \otimes_{A_*} R_* \otimes_{A_*} M_*)/\sim$ is naturally identified with $eM_* f \otimes_\Field eM_* f$.
\end{enumerate}
Under the identification (1), any basis for $e M_* f$ determines a basis $B = \{r_i \otimes m_i\}$ for $R_* \otimes M_*$ with the required property.
\end{proof}

Since $\tau^2 = 1$, and we are working over $\Field$, $(1+\tau)^2 = 0$.  We may therefore consider the bicomplex
\[
\xymatrix{
\dots & \HC_*(M\DTP M)\ar[l]_(.7){1+\tau}\ar@(ul,ur)^{\bdy_{\HC}} & \HC_*(M\DTP M)\ar[l]_{1+\tau}\ar@(ul,ur)^{\bdy_{\HC}}
& \HC_*(M\DTP M)\ar[l]_{1+\tau} \ar@(ul,ur)^{\bdy_{\HC}}&\dots\ar[l]_(.3){1+\tau}
}
\]
We denote this bicomplex by $\HC_{*,*}^\Tate(M\DTP M)$.  That is, $\HC_{p,q}^\Tate = \HC_q(M \DTP M)$, the vertical differential is $\bdy_{\HC(M \DTP M)}$, and the horizontal differential is $(1+ \tau)$.  We have two spectral sequences associated to $\HC^\Tate$, which we denote by $\hvE$ and $\vhE$.

\begin{proposition}
\label{prop:converge}
Suppose that $A$ is homologically smooth and $M$ is bounded.  The spectral sequences $\hvE^r_{p,q}$ and $\vhE^r_{p,q}$ attached to the bicomplex $\HC^\Tate(A, M \DTP M)$ converge to the homology of the total complex of $\HC^\Tate$.
\end{proposition}

\begin{proof}
As $\HC_*(M\DTP M)$ is bounded, the Hochschild-Tate bicomplex has $\HC_{p,q}^{\Tate} = 0$ for all but finitely many $q$.  The proposition therefore follows from Proposition~\ref{prop:demoted}.
\end{proof}

In the rest of this section we focus on the spectral sequence $\hvE$.  We will see that the differentials in $\hvE$ are natural operations on $\HH_*(M)$.  

Suppose $\xi \in \HC_k(M)$.  
Then we can write $\xi$ as a linear combination of pure tensors $r \otimes m$, i.e.
\[
\xi = \sum_\ell c_\ell r_\ell \otimes m_\ell
\]
with $c_\ell \in \Field$, $r_\ell \in R_{i_\ell}$, $m_\ell \in M_{j_\ell}$, and $i_\ell + j_\ell = k$.  The sum
\[
\xi^{\otimes 2} = \sum_\ell c_\ell^2 r_\ell \otimes (m_\ell \otimes r_\ell \otimes m_\ell)
\]
is not well-defined (it depends on $c_\ell$, $r_\ell$, $m_\ell$).  However, 

\begin{proposition}
The sum $\xi^{\otimes 2}$ is well-defined modulo the image of $(1+\tau)$.
\end{proposition}
\begin{proof}
This follows from the following computations:
\begin{eqnarray*}
(a r) \otimes (m \otimes (ar) \otimes m) & = & r \otimes (ma \otimes r \otimes ma) \text{   in $\HC(A,M \DTP M)$} \\
(r_1 + r_2) \otimes (m \otimes (r_1 + r_2) \otimes m) & = & r_1 \otimes (m \otimes r_1 \otimes m) + r_2 \otimes (m \otimes r_2 \otimes m) \\
& & + (1+\tau)(r_1 \otimes (m \otimes r_2 \otimes m))\\
r \otimes ((m_1 + m_2) \otimes r \otimes (m_1 + m_2)) & = & r \otimes (m_1 \otimes r \otimes m_1) + r \otimes (m_2 \otimes r \otimes m_2) \\
& & + (1+\tau)(r \otimes (m_1 \otimes r \otimes m_2))
\end{eqnarray*}
\end{proof}

We will use the operation $\xi \mapsto \xi^{\otimes 2} + \Image(1+\tau)$ to study $\hvE$:

\begin{proposition}
  \label{prop:Frob}
  Let $A$ be a \dg algebra and let $M$ be a \dg bimodule for $A$.  For each $k$, the assignment $\xi \mapsto \xi^{\otimes 2} + \Image(1+\tau)$ is a $\Field$-linear isomorphism of $\HC_k(M)$ onto $\hvE^1_{p,2k}$.
  Moreover, $\hvE^1_{p,2k+1} = 0$.
\end{proposition}

\begin{proof}
It is clear that $\xi^{\otimes 2} \in \ker(1+\tau)$, so that we do
have a well-defined map from $\HC_k(M)$ to $\hvE^1_{p,2k} :=
\ker(1+\tau)/\Image(1+\tau)$.  

Let us show that the map is linear.
Roughly speaking, we show that for $\phi$ and $\psi$ in $\HC_k$, $(\phi + \psi)^{\otimes 2} - \phi^{\otimes 2} - \psi^{\otimes 2} =  \phi \otimes \psi + \psi \otimes \phi$, where the right hand side is the image of $(1+\tau)$ under $\phi \otimes \psi$.  More precisely, if $\phi =  \sum_\ell c_\ell r_{\ell} \otimes m_{\ell}$ and $\psi = \sum_{\lambda} b_{\lambda} s_{\lambda} \otimes n_{\lambda}$, then one computes
\[
(\phi + \psi)^{\otimes 2} - \phi^{\otimes 2} - \psi^{\otimes 2} = 
(1+\tau)\Bigl( \sum_{\ell_1,\lambda_2} c_{\ell_1}b_{\lambda_2} r_{\ell_1} \otimes (m_{\ell_1}\otimes r_{\lambda_2} \otimes m_{\lambda_2})\Bigr).
\]
To show that the map $\HC_k(M) \to \hvE^1_{p,2k}$ is an isomorphism,
choose a basis $\{r_\ell \otimes m_\ell\}_{\ell \in L}$ and
$\{r_\ell \otimes (m_\ell \otimes r_\lambda \otimes
m_\lambda)\}_{(\ell,\lambda) \in L \times L}$ for $\HC(M)$ and $\HC(M
\DTP M)$ as in Proposition~\ref{prop:Z2action}.  As the basis of $\HC(M\DTP M)$ is stable for the
$\ZZ/2$-action, we may use it to construct a basis for $\ker(1+\tau)$
and for $\Image(1+\tau)$.  A basis element for $\ker(1+\tau)$ has one
of the following two forms:
\begin{enumerate}
\item $r_\ell \otimes (m_\ell \otimes r_\ell \otimes m_\ell)$, i.e.\ the image of $\xi = r_\ell \otimes m_\ell$ under $\xi \mapsto \xi^{\otimes 2}$
\item $r_\ell \otimes (m_\ell \otimes r_\lambda \otimes m_\lambda) + r_\lambda \otimes (m_\lambda \otimes r_\ell \otimes m_\ell)$ for $\ell \neq \lambda$.
\end{enumerate}
Just the elements of form (2) are a basis for $\Image(1+\tau)$.  Thus the images of the elements of form (1) in $\ker(1+\tau)/\Image(1+\tau) = \hvE^1$ form a basis.  The map $\xi \mapsto \xi^{\otimes 2} + \Image(1+\tau)$ is a bijection on these bases, and is therefore an isomorphism. 

Finally, note that in odd gradings, there are no elements of the form (1), so elements of the form (2) span. Since these are in the image of $1+\tau$, it follows that $\hvE^1_{p,2k+1} = 0$.
\end{proof}

\begin{remark}
\label{rem:Frob}
If we were working not with $\Field$ but with a larger field of characteristic 2, the map of Proposition \ref{prop:Frob} would be ``Frobenius-linear,'' i.e. $(c \xi)^{\otimes 2} = c^2 (\xi)^{\otimes 2}$.  As $c \mapsto c^2$ is a field homomorphism (\resp isomorphism) for any field (\resp perfect field) of characteristic 2, another way to express this is to say that the map induces a linear isomorphism from the \emph{Frobenius twist} of $\HC_k(M)$ to $\hvE^1_{p,2k}$.  
\end{remark}

Since $\hvE^1_{p,q} = 0$ for $q$ odd, the differential on $\hvE^1_{p,q}$ must vanish and we have $\hvE^1_{p,q} = \hvE^2_{p,q}$.

\begin{proposition}
\label{prop:TateE2}
Let $d^2:\hvE^2_{p,q} \to \hvE^2_{p+1,q-2}$ denote the differential on the second page of the spectral sequence.  For each $p$ and each $k$, the following diagram commutes:
\[
\xymatrix{
\ar[rr]^{\xi \mapsto \xi^{\otimes 2} + \Image(1+\tau)} \HC_k(M) \ar[d]_{\partial_{\HC}} & & \hvE^1_{p,2k} \ar@{=}[r] & \hvE^2_{p,2k} \ar[d]^{\hvd^2}  \\
\ar[rr]_{\xi \mapsto \xi^{\otimes 2} + \Image(1+\tau)}
\HC_{k-1}(M) & &  \hvE^1_{p+1,2k-1}  \ar@{=}[r] & \hvE^2_{p+1,2k-2}.
}
\]
\end{proposition}

\begin{proof}
It suffices to prove that  
\begin{equation}
\label{eq:TateE2.1}
\hvd^2\left(\xi^{\otimes 2} + \Image\left(1+\tau\right)\right) = \left(\partial_{\HC(M)}\left(\xi\right)\right)^{\otimes 2} + \Image\left(1+\tau\right)
\end{equation}
when $\xi$ is of the form $r \otimes m$, as these terms generate $\HC_k(M)$.  In that case $\xi^{\otimes 2} = r \otimes (m \otimes r \otimes m)$, and
\[
\partial_{\HC(M \DTP M)} \left(r \otimes \left(m \otimes r \otimes m\right)\right)= \left(1+\tau\right) \left( \partial\left(r\right) \otimes \left(m \otimes r \otimes m\right) + r \otimes \left(\partial(m) \otimes r \otimes m\right)\right)
\]
It follows that $\bigl(\xi^{\otimes 2}, \partial(r) \otimes (m \otimes r \otimes m)+ r \otimes (\partial(m) \otimes r \otimes m)\bigr)$ is an hv sequence (Remark \ref{rem:hv}) of length 1, so that 
\begin{equation}
\label{eq:TateE2.3}
\hvd^2\left(\xi^{\otimes 2}\right) = \partial_{\HC(M \DTP M)} \left(\partial\left(r\right) \otimes \left(m \otimes r \otimes m\right)+ r \otimes \left(\partial\left(m\right) \otimes r \otimes m\right)\right)
\end{equation}
Expanding the right hand sides of \eqref{eq:TateE2.1} and \eqref{eq:TateE2.3} completes the proof.
\end{proof}
 
It follows that $\hvE^3_{p,2k}$ is naturally identified with $\HH_k(M)$.  Since $\hvE^3_{p,2k+1} = 0$, we have $d^3 = 0$ and in fact $d^{2i+1} = 0$ for every $i$.

\begin{definition}\label{def:pi-formal-bimod}
  The bimodule $M$ is \emph{$\pi$-formal} if the operation $d^{2i}$ induced
  by the spectral sequence $\hvE$ vanishes for each $i \geq
  2$. (Equivalently, $M$ is $\pi$-formal if the spectral sequence
  $\hvE$ collapses at the $E^3$ page.)
\end{definition}

Now that Theorem~\ref{thm:hoch-local} has been formulated precisely,
we can also prove it.

\begin{proof}[Proof of Theorem~\ref{thm:hoch-local}]
  Suppose $A$ is homologically smooth and proper and that $M$ is a
  $\pi$-formal $(A,A)$-bimodule.  By Proposition \ref{prop:converge},
  the two spectral sequences $\vhE$ and $\hvE$ attached to the
  Hochschild-Tate bicomplex for $M$ converge to the same group $V$.
  Since the vertical differentials in the bicomplex are the Hochschild
  differentials for $M \DTP M$, we have $\vhE^1_{p,q} = \HH_q(A,M \DTP
  M)$, verifying assertion~(\ref{item:hoch-local-E1}) of the theorem.
  By the definition of $\pi$-formality, the spectral sequence $\hvE$
  degenerates at $\hvE^3$, i.e. $\hvE^3 = \hvE^\infty$ is the
  associated graded of a filtration $\vF$ on $V$.  By Proposition
  \ref{prop:TateE2} we have $\hvE^3_{p,2q} = \HH_{q}(A,M)$ and
  $\hvE^3_{p,2q+1} = 0$, verifying
  assertion~(\ref{item:hoch-local-Einfty}) of the theorem.
\end{proof}

\subsection{Naturality and \texorpdfstring{$\pi$}{pi}-formality}\label{sec:natural-formal}
\begin{theorem}
\label{thm:representable}
Suppose that $A$ is homologically smooth and proper, and let $A^!$ be the bimodule of Proposition \ref{prop:cobar}.  The following are equivalent:
\begin{enumerate}
\item\label{item:rep-1} Every \dg bimodule over $A$ is $\pi$-formal
  (Definition~\ref{def:pi-formal-bimod}).
\item\label{item:rep-2} The \dg bimodule $A^!$ is $\pi$-formal.
\item\label{item:rep-3} For each $i \geq 2$, the element $1 \in \Hom(A^!,A^!) \cong \HH_0(A,A^!)$ is killed by $d^{2i}:\HH_0(A,A^!) \dashrightarrow \HH_{- i}(A,A^!)$.
\end{enumerate}
\end{theorem}

\begin{proof}
It is clear that (\ref{item:rep-1}) implies (\ref{item:rep-2}) and that (\ref{item:rep-2}) implies (\ref{item:rep-3}).  Let us show that (\ref{item:rep-3}) implies (\ref{item:rep-1}).

A map $f\co M \to N$ of \dg bimodules induces a map $f \DTP f:M \DTP M \to N \DTP N$, which in turn induces a map $f\co \HC^\Tate(M
\DTP M) \to \HC^\Tate(N \DTP N)$ of Hochschild-Tate bicomplexes, so
that the differentials in $\hvE$ are natural with respect to maps in
$\Bimod{A}$.
If $f$ is a quasi-isomorphism, then by Proposition \ref{prop:TateE2},
$f$ induces an isomorphism $\hvE^r_{p,q}(M) \to \hvE^r_{p,q}(N)$ for
$r \geq 3$.  Thus for $r \geq 3$, the differentials in $\hvE^r$ are
natural with respect to maps in $D(\Bimod{A})$.

By Proposition \ref{prop:cobar} and the Yoneda lemma, $d^4:\HH_i(M) \to \HH_{i - 2}(M)$ is given by precomposition with an element of $\Hom(\Sigma^{i-2} A^!,\Sigma^i A^!) \cong \Hom(A^!,\Sigma^2 A^!)$---in fact this element is $d^4(1)$.  Thus if $d^4(1) = 0$, $d^4 = 0$ for every bimodule $M$.  In that case $\hvE^6 = \hvE^5 = \hvE^4 = \HH_i(M)$ and an identical argument shows that $d^6:\HH_i(M) \to \HH_{i-3}(M)$ vanishes so long as $d^6(1)$ vanishes.  The evident induction completes the proof.
\end{proof}

\begin{definition}\label{def:pi-formal}
  If $A$ satisfies the (equivalent) conditions of
  Theorem~\ref{thm:representable} then we say that $A$ is
  \emph{$\pi$-formal}.
\end{definition}

\subsection{\texorpdfstring{$\pi$}{pi}-formal and neutral bimodules}\label{sec:neutral}

In this section, $A$ is a homologically smooth and proper \dg algebra over $\Field$.  Let $A^!$ be the bimodule of Proposition~\ref{prop:cobar}, so that for every \dg bimodule $M$ we have an identification
\[
\Hom(\Sigma^j A^!,M) \cong \HH_j(M),
\]
where $\Hom$ denotes the morphisms in the derived category of $(A,A)$ bimodules.
Let us define Hochschild cohomology as usual by $\HH^j(M) =
\Hom(\Sigma^{-j} A, M)$.  Then any map $f:\Sigma^{d} A^! \to A$,
i.e.\ any element of $\HH_{d}(A)$, induces a map 
\[
f^*\co \HH^k(M) \to \HH_{d-k}(M)
\]
by precomposition.

\begin{definition}\label{def:neutral}
  We call a bimodule $M$ \emph{$d$-neutral} if there is class $f
  \in \HH_d(A)$ such that the induced map $f^*\co \HH^k(M) \to
  \HH_{d-k}(M)$ is an isomorphism for every $k$.  We say that $M$ is
  \emph{neutral} if $M$ is $d$-neutral for some $d$.  We call $f$ the
  \emph{neutralizing element}.
\end{definition}

\begin{remark}
Suppose that there is an isomorphism of bimodules $A^! \otimes_A M \cong M$, and that this isomorphism is witnessed by a map $f:A^! \to A$.  In other words suppose that the composite map
\[
A^! \DTP M \stackrel{f\otimes \Id}{\longrightarrow} A \DTP M = M
\]
is an isomorphism.  Then the induced map $\Hom(\Sigma^k A^!, A^! \DTP
M)\to \Hom(\Sigma^k A^!,M) $ is also an isomorphism.  Furthermore, we
may identify $\Hom(\Sigma^k A^!,A^! \DTP M)$ with $\Hom(\Sigma^k
A,M)$, and the map $\Hom(\Sigma^k A,M)\to \Hom(\Sigma^k A^!,M)$
coincides with the map induced by $f:A^! \to A$.  Using the
identification of Proposition \ref{prop:cobar}, we see that $M$ is
$0$-neutral and $f$ is a neutralizing element.
\end{remark}

The relevance of neutrality to this paper is the following:
\begin{proposition}\label{prop:neutral-d2i}
  Suppose that the operations $d^{2i}$ on $\HH_*(A)$ vanish for all
  $i\geq 2$. Then any neutral bimodule is $\pi$-formal.
\end{proposition}
\begin{proof}
  This follows from a short Yoneda-style argument.
  Fix a neutral bimodule $M$ with neutralizing element $f\in
  \Hom(\Sigma^dA^!,A)=\HH_d(A)$. 
  Suppose
  $\alpha\in \HH_k(M)=\Hom(\Sigma^kA^!,M)$.  Let
  $\beta\in\Hom(\Sigma^{k-d}A,M)$ be $(f^*)^{-1}(\alpha)$. Then
  $\alpha = \beta_*(f)$ (where $\beta_*$ denotes post-composition by
  $\beta$). By naturality of $d^{2i}$,
  $d^{2i}(\alpha)=\beta_*(d^{2i}(f))$. But by hypothesis,
  $d^{2i}(f)=0$.
\end{proof}

\begin{corollary}\label{cor:neutral-formal}
  If $\HH_*(A)$ is supported in a single grading then any neutral
  $(A,A)$-bimodule is $\pi$-formal.
\end{corollary}

\begin{remark}
If $A$ is Calabi-Yau of dimension $d$ (that is, if there is a quasi-isomorphism $\Sigma^d A^! \cong A$), then every bimodule is $d$-neutral.
A partial converse holds: if the diagonal bimodule is $d$-neutral, then by definition for every $M$, there is a map $f_M:\Sigma^d A^! \to A$ inducing an isomorphism $f_M^*:\Hom(A,M) \to \Hom(\Sigma^d A^!,M)$.  If $f_M$ can be chosen independent of $M$, then Yoneda's lemma implies that the map $\Sigma^d A^! \to A$ is also a quasi-isomorphism.  
\end{remark}

\begin{remark}
  Suppose that $X$ is a smooth, projective, $d$-dimensional algebraic variety with canonical bundle $\omega_X$.  An argument due to van den Bergh and Bondal  (cf.~\cite[Example 8.1.4]{KontsevichSoibelman06:KoSo06}) shows that the derived category of coherent sheaves on $X$ is equivalent to the derived category of left dg modules over a homologically smooth and proper dg algebra $A$.  Under this dictionary, $\Bimod{A}$ is identified with the derived category of coherent sheaves on $X \times X$, and $\HH_d(A)$ is identified with $H^0(X,\omega_X)$.  If $\cF$ is an object of this derived category corresponding to a bimodule $M$, the map $\HH^k(M) \to \HH_{d-k}(M)$ induced by an element  $f \in \HH_d(A)$ is identified with the map
\begin{equation}
\label{eq:claudius}
\Hom(\Delta_* \cO_X,\cF)   \to \Hom(\Delta_* \omega_X^{-1},\cF)
\end{equation}
induced by a section of $\omega_X$.  Here $\Delta:X \to X \times X$ denotes the diagonal map.  Using the right adjoint $\Delta^!$ to $\Delta_*$, one may rewrite \eqref{eq:claudius} as 
\begin{equation}
\mathbf{R}\Gamma(X;\Delta^! \cF) \to \mathbf{R}\Gamma(X,\omega_X \otimes \Delta^! \cF)
\end{equation}
In particular, if $X$ has an effective canonical divisor $D$ (for instance, if $X$ is of general type), a sufficient condition for $\cF$ to be $d$-neutral is for the restriction of $\cF$ to the diagonal copy of $X$ to be supported away from $D$.
\end{remark}

\subsection{Integral models and \texorpdfstring{$\pi$}{pi}-formality.}
\label{subsec:imapf}
In this section we show that the existence of an integral lift of $A$
implies vanishing of the operations $d_{2i}$ for $i$ even. While we
will not use this result in the rest of the paper, it seems
likely that the bordered algebras do have integral lifts.

Let $A_\ZZ$ be a homologically smooth and proper \dg algebra over $\ZZ$, with resolution $R_\ZZ \to A_\ZZ$. 
We make the following additional assumptions:
\begin{enumerate}
\item The underlying graded group $A_{\ZZ,*}$ of $A_\ZZ$ is free abelian
\item The underlying $(A_{\ZZ,*},A_{\ZZ,*})$-bimodule $R_{\ZZ,*}$ of $R_{\ZZ}$ is a direct sum of bimodules of the form $e A_{\ZZ,*} \otimes_{\ZZ} A_{\ZZ,*} f$, where $e$ and $f$ are idempotents in $A_{\ZZ,*}$
\end{enumerate}
Let $A^!$ be the cobar bimodule of Proposition \ref{prop:cobar}.  Let $A_\Field$ denote the reduction of $A_\ZZ$ mod 2, and $A^!_\Field = A_\Field \otimes_{A_\ZZ} A^!$.  We will study the Hochschild complex $\HC(A_{\ZZ},A^! \DTP A^!)$ and its relation to $\HC(A_\Field,A^!_\Field \DTP A^!_\Field)$.

\begin{proposition}
\label{prop:Z2actionZZ}
The map $\tau:\HC(A_\ZZ,A^!\DTP A^!) \to \HC(A_\ZZ, A^!\DTP A^!)$ that sends $r \otimes (m \otimes r' \otimes m')$ to $(-1)^{(|r| + |m|)(|r'|+|m'|)} r' \otimes (m' \otimes r \otimes m)$ is a map of chain complexes and satisfies $\tau \circ \tau(x) = x$.  Moreover, there is a $\ZZ$-basis of $\HC(A_\ZZ,A^!)$ of the form $\{r_i \otimes x_i\}$ such that $\{r_i \otimes x_i \otimes r_j \otimes x_j\}$ is a $\ZZ$-basis for $\HC(A_\ZZ,A^! \DTP A^!)$. 
\end{proposition}

The proof, which uses our assumption (2) above, is the same as the proof of Proposition~\ref{prop:Z2action}.

We have the following variant of the Hochschild-Tate bicomplex of Section \ref{subsec:HHTate}:
\[
\xymatrix{
\dots & \HC_*(A^!\DTP A^!)\ar[l]_(.7){1+\tau}\ar@(ul,ur)^{-\bdy_{\HC}} & \HC_*(A^!\DTP A^!)\ar[l]_{1-\tau}\ar@(ul,ur)^{\bdy_{\HC}}
& \HC_*(A^!\DTP A^!)\ar[l]_{1+\tau} \ar@(ul,ur)^{-\bdy_{\HC}}&\dots\ar[l]_(.3){1-\tau}
}
\]
  The groups have $\HC^\Tate_{p,q} = \HC_q(A^! \DTP A^!)$, but the differentials depend on the parity of $p$.  (The alternating signs in front of $\bdy_{\HC}$ give us $d^h d^v + d^v d^h = 0$.)  We denote this bicomplex by $\HC^\Tate(A_\ZZ,A^! \DTP A^!)$. The integral Hochschild-Tate complex is a bicomplex of free abelian groups; note that reducing it mod 2 gives the definition of $\HC^\Tate(A_\Field,A^!_\Field \DTP A^!_\Field)$ of the previous section.

The horizontal homology of this integral Hochschild-Tate complex has the following vanishing pattern:

\begin{proposition}
We have $\hvE^1_{p,q} = 0$ in the following cases:
\begin{enumerate}
\item\label{item:vanish-odd} $q$ is odd
\item\label{item:vanish-0} $q = 0$ mod 4 and $p$ is odd.
\item\label{item:vanish-2} $q = 2$ mod 4 and $p$ is even.
\end{enumerate}
\end{proposition}

\begin{remark}
The possible nonvanishing groups in $\hvE^1_{p,q}$ are the dots in the following diagram:
\[
\begin{array}{cccccc}
\bullet & 0 & \bullet & 0 & \bullet & 0 \\
0 & 0 & 0 & 0 & 0 & 0 \\
0 & \bullet & 0 & \bullet & 0 & \bullet \\
0 & 0 & 0 & 0 & 0 & 0 \\
\bullet & 0 & \bullet & 0 & \bullet & 0 \\
0 & 0 & 0 & 0 & 0 & 0 \\
0 & \bullet & 0 & \bullet & 0 & \bullet 
\end{array}
\]
\end{remark}

\begin{proof}
Let $\{r_i \otimes (x_i \otimes r_j \otimes x_j)\}_{(i,j) \in I \times I}$ be a basis for $\HC(A,A^! \DTP A^!)$ as in Proposition \ref{prop:Z2actionZZ}.  Then $\HC_q(A^! \DTP A^!)$ is spanned by those basis elements with $|r_i| + |x_i| + |r_j| + |x_j| = q$.  

If $q$ is odd, then this subset of basis elements contains nothing of the form $r_i \otimes (x_i \otimes r_i \otimes x_i)$.  It follows that $\HC_q(A^! \DTP A^!)$ is a free $\ZZ[\ZZ/2]$-module.  Because of this, $\ker(1-\tau)/\Image(1+\tau)$ and $\ker(1+\tau)/\Image(1-\tau)$ both vanish---this proves assertion~(\ref{item:vanish-odd}).

Suppose now that $q$ is even.  Then we may write $\HC_q(A^! \DTP A^!)$ as a sum of a free $\ZZ[\ZZ/2]$-module (spanned by basis elements of the form $r_i \otimes (x_i \otimes r_j \otimes x_j)$ for $i \neq j$) and the module spanned by elements of the form $r_i \otimes (x_i \otimes r_i \otimes x_i)$.  We have
\begin{align*}
\tau(r_i \otimes (x_i \otimes r_i \otimes x_i)) &= (r_i \otimes (x_i \otimes r_i \otimes x_i)) & \text{if $q/2 = |r_i| + |x_i|$ is even} \\
\tau(r_i \otimes (x_i \otimes r_i \otimes x_i)) &= -(r_i \otimes (x_i \otimes r_i \otimes x_i)) & \text{if $q/2 = |r_i| + |x_i|$ is odd}.
\end{align*}
In other words, if $q$ is divisible by 4, then $\HC_q(A^! \DTP A^!)$ is a sum of a free $\ZZ[\ZZ/2]$-module and a trivial module on which $\tau$ acts by the scalar $1$.  On the other hand if $q$ is congruent to 2 mod 4 then $\HC_q(A^! \DTP A^!)$ is a sum of a free module and a module on which $\tau$ acts by the scalar $-1$.  In the former case $\ker(1+\tau)/\Image(1-\tau)$ vanishes and in the latter case $\ker(1-\tau)/\Image(1+\tau)$ vanishes.
\end{proof}

\begin{corollary}
Let $A$ be an $\Field$ \dg algebra that is homologically smooth and
proper, and suppose that $A$ arises as the mod 2 reduction of a \dg
algebra $A_\ZZ$ satisfying the conditions (1) and (2) above.  Then the
operations $d^{2r}$ vanish for $r \equiv 0\pmod{2}$.
\end{corollary}

\subsection{Relation with the Hochschild-to-cyclic spectral sequence}\label{subsec:is-hodge-to-derham}

\subsubsection{Cyclic modules and the Hodge-to-de Rham spectral sequence}
Let $\bLam$ be Connes's cyclic category, and let $\cM:\bLam^{\op} \to \Field\text{\emph{-vector spaces}}$ be a cyclic module over $\Field$.  Thus, $\cM$ is given by the following data:
\begin{enumerate}
\item A sequence of vector spaces $\cM_n$, $n \in \ZZ_{\geq 0}$
\item Face and degeneracy maps $d_i:\cM_n \to \cM_{n-1}$ and $s_i:\cM_n \to \cM_{n+1}$ for $i = 0,\ldots,n$.
\item A morphism $t_n:\cM_n \to \cM_n$ that generates an action of $\ZZ/(n+1)$ on $\cM$.
\end{enumerate}
These maps are subject to additional relations. 
See for instance \cite[Section 2.5]{LodayBook} for details.  We let $\Cyc(\Field)$ denote the category of cyclic $\Field$-modules.  A cyclic module $\cM$ has an underlying simplicial module, from which we may extract a chain complex in the usual way.  We denote this chain complex by $(\HC(\cM),\bdy_{\HC(\cM)})$ and its homology by $\HH(\cM)$.  Thus, $\HC_n(\cM) = \cM_n$ and the differential is given by
\[
\bdy(x) = \sum_{i = 0}^n d_i(x) \text{ for $x \in \HC_n(\cM)$}.
\]
A map $\cM \to \cN$ of complexes of cyclic modules is called a quasi-isomorphism if it induces a quasi-isomorphism $\HC(\cM) \to \HC(\cN)$.  We let $h\Cyc(\Field)$ denote the localization of $\Cyc(\Field)$ with respect to quasi-isomorphisms.

\begin{remark}
Our usage of $\HC$ does not agree with that of \cite{LodayBook}, where it is used to denote cyclic homology.  We will denote cyclic homology by $\CH$ instead.
\end{remark}

We may also attach to $\cM_*$ the ``cyclic bicomplex'' $\mathit{CC}(\cM)$, which looks like this
\[
\xymatrix{
& \vdots \ar[d] & \vdots \ar[d] & \vdots \ar[d] \\
0 & \ar[l] \HC_2(\cM)\ar[d]_{\bdy_{\HC}} & \ar[l]_{1 - t_2} \HC_2(\cM) \ar[d]_{b'} & \ar[l]_{N} \HC_2(\cM) \ar[d]_{\bdy_{\HC}} & \cdots \ar[l]_{\quad 1-t_2} \\
 0 & \ar[l] \HC_1(\cM) \ar[d]_{\bdy_{\HC}} & \ar[l]_{1-t_1} \HC_1(\cM) \ar[d]_{b'} & \ar[l]_N \HC_1(\cM) \ar[d]_{\bdy_{\HC}} & \cdots \ar[l]_{\quad 1-t_1} 
\\0 & \ar[l] \HC_0(\cM) \ar[d] & \ar[l]_{1-t_0} \HC_0(\cM)\ar[d] &\ar[d] \ar[l]_N \HC_0(\cM) &\cdots \ar[l]_{\quad 1-t_0} \\
& 0 & 0 & 0
}
\]
where for $x \in \HC_n(\cM)$, the maps $b'$ and $N$ are given by
\[
b'(x)  =  \sum_{i = 0}^{n-1} d_i(x)  \qquad\qquad
N(x)  =  \sum_{i = 0}^n t_n^i(x).
\]
The odd columns of this complex are acyclic.

\begin{remark}
The nerve of the category $\bLam$ is homotopy equivalent to the classifying space of the circle group, and because of this cyclic modules are good models for homotopy local systems on the classifying space of the circle $B\mathrm{U}(1)$ \cite{DwyerHopkinsKan}.  The complex $\mathit{CC}(\cM)$ computes the homology of $B\mathrm{U}(1)$ with coefficients in this local system.
\end{remark}

\begin{remark}
\label{rem:cyclicspheres}
An example of the previous remark is the following construction of
\cite[Section 7.1--7.2, Exercise 7.2.2]{LodayBook}.  If $X$ is a pointed space with a $\mathrm{U}(1)$-action then there is a cyclic module $\Field[X]$ with the following properties:
\begin{enumerate}
\item $\HH_*(\Field[X])$ is naturally isomorphic to the reduced homology $\widetilde{H}_*(X,\Field)$.
\item $\CH_*(\Field[X])$ is naturally isomorphic to the reduced equivariant homology $\widetilde{H}_*^{\mathrm{U}(1)}(X,\Field)$.
\end{enumerate}
(What we call $\Field[X]$, Loday denotes by $S_.(X)$, reflecting its construction as a variant of the singular chain complex.)  In particular if $X = S^n$ is an $n$-dimensional sphere carrying the trivial action of $\mathrm{U}(1)$, then $\HH_*(\Field[X])$ is concentrated in degree $n$, and $\CH_m(\Field[X]) = H_{m-n}(B\mathrm{U}(1),\Field)$.  The object $\Field[S^n]$ represents the functor $\CH_n(\cM)$ in the homotopy category $h\Cyc(\Field)$: we have $\Hom_{h\Cyc(\Field)}(\Field[S^n],\cM) \cong \CH_n(\cM)$.
\end{remark}

The \emph{Hochschild-to-cyclic spectral sequence}, also called the \emph{Hodge-to-de Rham spectral sequence}, is the spectral sequence $\vhE$ corresponding to this bicomplex.  We have
\[
\vhE^1_{pq}(\cM) =\vhE^2_{pq}(\cM) =  
\begin{cases}
\HH_q(\cM) & \text{if $p$ is even and $\geq 0$} \\
0 & \text{otherwise}
\end{cases}
\]
It is a first-quadrant spectral sequence converging to $\CH_{p+q}(\cM)$, the total homology of the bicomplex $\mathit{CC}(\cM)$.  A map $f:\cM \to \cN$ of cyclic modules induces a map $\vhE^r_{pq}(\cM) \to \vhE^r_{pq}(\cM)$ of spectral sequences, and if $f$ is a quasi-isomorphism then the induced map is an isomorphism for $r \geq 1$.  Thus the Hodge-to-de Rham spectral sequence is functorial for maps in $h\Cyc(\Field)$.

\begin{proposition}
\label{prop:cyclicformality}
Let $\cM$ be a bounded cyclic module, i.e.\ a cyclic module with $\HH_n(\cM) = 0$ for all but finitely many $n$.  Then the following are equivalent:
\begin{enumerate}
\item The Hodge-to-de Rham spectral sequence for $\cM$ collapses at $E^1$.
\item There is a quasi-isomorphism $\cM \cong \bigoplus_{j = 0}^k \cN_j$ where each $\cN_j$ has $\HH_n(\cN_j) = 0$ for all but one value of $n$.
\end{enumerate}
\end{proposition}

\begin{proof}
Let us show that (2) is a consequence of (1)---the reverse implication
is trivial.

We will prove that if the Hodge-to-de Rham spectral sequence for $\cM$ collapses at $E^1$ then $\cM$ is a direct sum (in $h\Cyc(\Field)$) of copies of the cyclic modules $\Field[S^\ell]$ of Remark \ref{rem:cyclicspheres}.  We will induct on the dimension $d$ of $\bigoplus_k \HH_k(\cM)$.  If $d = 1$ and $\HH_k(\cM) = \Field$, then $\CH_k(\cM) = \Field$ as well and the representing map $\Field[S^k] \to \cM$ is a quasi-isomorphism.  Suppose now that the assertion has been proved for all $\cM''$ with $\dim(\bigoplus_k \HH_k(\cM'')) < d$.

For the inductive step we need the following claim: the obstructions to splitting a short exact sequence of cyclic modules  $\Field[S^j] \to \cE \to \Field[S^k]$ are the nontrivial differentials in the Hodge-to-de Rham spectral sequence of $\cE$.  More precisely, let $\cE$ be a cyclic module and suppose we have maps $\Field[S^{j}] \to \cE \to \Field[S^{k}]$ that induce short exact sequences of (bi)complexes
\begin{eqnarray*}
0 \to \HC(\Field[S^{j}]) \to \HC(\cE) \to \HC(\Field[S^{k}]) \to 0 \\
0 \to \mathit{CC}(\Field[S^j]) \to \mathit{CC}(\cE) \to \mathit{CC}(\Field[S^k]) \to 0
\end{eqnarray*}
The $\vhE^r$ spectral sequence attached to the bicomplex $\cE$ is supported in rows $j$ and $k$.  The differential $d^{j - k + 1}:\HH_k(\cE) \to \HH_j(\cE)$ determines the connecting homomorphism in the long exact sequence
\[
\CH_k(\Field[S^j]) \to \CH_k(\cE) \to \CH_k(\Field[S^k]) \stackrel{\delta}{\to} \CH_{k-1}(\Field[S^j])
\]
In particular, if $\vhE^r$ degenerates at $r = 1$, then this connecting homomorphism is zero.  It follows that under this degeneration hypothesis the map
\[
\Hom_{h\Cyc(\Field)}(\Field[S^k],\cE) \to \Hom_{h\Cyc(\Field)}(\Field[S^k],\Field[S^k])
\]
is surjective, or in other words that $\cE = \Field[S^j] \oplus \Field[S^k]$ in $h\Cyc$.

Now let us return to $\cM$.  Let $j$ denote the smallest number for which $\CH_j(\cM)$ is nonzero.
Let $\cM'$ denote the direct sum of $\dim(\CH_j(\cM))$ many copies of
$\Field[S^j]$.  Then after replacing $\cM$ with a quasi-isomorphic cyclic module if
necessary there is a short exact sequence of cyclic modules
$\cM' \to \cM \to \cM''$ such that $\HH_j(\cM') \to \HH_j(\cM)$ is an
isomorphism.  From the long exact sequence attached to $0 \to
\HC(\cM') \to \HC(\cM) \to \HC(\cM'') \to 0$, it follows that
$\HH_{k}(\cM) \to \HH_k(\cM'')$ is an isomorphism for $k > j$.  The
associated map of spectral sequences $\vhE^r_{pq}(\cM) \to
\vhE^r_{pq}(\cM'')$ is an isomorphism for $q > j$, which is where
$\vhE^r_{pq}(\cM'')$ is supported, and the differentials in
$\vhE^r_{pq}(\cM'')$ must vanish.  By the inductive hypothesis $\cM''$
is quasi-isomorphic to a direct sum of copies of $\Field[S^{\ell_k}]$,
$\ell_k > j$.  The Proposition is now a consequence of the claim above.
\end{proof}

\subsubsection{The Hochschild-Tate bicomplex of a cyclic module}

There is an operation of restriction from local systems on $B\mathrm{U}(1)$ to local systems on $B\ZZ/2$.  In this section we model this operation at the level of cyclic modules.  Suppose that $\cM$ is a cyclic module.  Then define 
\[
\TT_n(\cM) = \bigoplus_{p = 0}^n \HC_{n+1}(\cM)
\]
and define $\bdy_{\TT_n(\cM)}$ as follows.  If $x \in \TT_n(\cM)$ belongs to the copy of $\HC_{n+1}(\cM)$ indexed by $p$, then set $q = n - p$.  If $p \neq 0$ and $q \neq 0$, then 
\[
\bdy_{\TT_n(\cM)} = \left(\sum_{i = 0}^{p} d_i\left(x\right) \right)_{(p-1,q)}+ \left(\sum_{j = p+1}^{p+q+1}d_j\left(x\right)\right)_{(p,q-1)}
\]
where the first sum belongs to the copy of $\HC_{n}(\cM)$ indexed by $(p-1,q)$ and the second sum belongs to the copy of $\HC_n(\cM)$ indexed by $(p,q-1)$.  If $p = 0$ then we omit the first sum from the definition of $\bdy_{\TT_n(\cM)}(x)$ and if $q = 0$ we omit the second sum.  (If $p = q = 0$, then $n=0$ and $\bdy_{\TT_0} = 0$.)

\begin{proposition}
\label{prop:TT}
$(\TT_*(\cM),\bdy_{\TT})$ is a chain complex (that is, $\bdy_\TT^2 = 0$), and it is naturally quasi-isomorphic to $\HC(\cM)$.
\end{proposition}

\begin{proof}
The complex
$\TT(\cM)$ is just the total complex of the double complex
\[
\xymatrix{
& \vdots \ar[d] & \vdots \ar[d] & \vdots \ar[d] \\
0 & \ar[l] \HC_3(\cM)\ar[d] & \ar[l] \HC_4(\cM) \ar[d] & \ar[l] \HC_5(\cM) \ar[d] & \cdots \ar[l] \\
 0 & \ar[l] \HC_2(\cM) \ar[d] & \ar[l] \HC_3(\cM) \ar[d]& \ar[l] \HC_4(\cM) \ar[d] & \cdots \ar[l]
\\0 & \ar[l] \HC_1(\cM)\ar[d] & \ar[l] \HC_2(\cM) \ar[d] & \ar[l] \HC_3(\cM) \ar[d] &\cdots \ar[l] \\
& 0 & 0 & 0 
}
\]
where the horizontal differential $\HC_{n+1}(\cM) \to \HC_n(\cM)$ in
the $q\th$ row is given by $\sum_{i = 0}^{n-q} d_i$ and the vertical
differential $\HC_{n+1}(\cM) \to \HC_n(\cM)$ in the $p\th$ column is
given by $\sum_{i = p+1}^{n+1} d_i$.  The standard simplicial
identities for the face maps imply that the horizontal and vertical
differentials commute and square to zero.  There is an augmentation
map from the bottom row of this bicomplex to $(\HC(\cM),\bdy_{\HC(\cM)})$ whose $n\th$ term  $\HC_{n+1}(\cM) \to \HC_n(\cM)$ is given by $d_n$.  To prove that this augmentation map induces a quasi-isomorphism from the total complex of $\TT(\cM)$ to $\HC(\cM)$, it suffices to show that the augmented columns are exact.  Indeed, the degeneracy map $s_p:\cM_{n+p} \to \cM_{n+p+1}$, regarded as a map $\HC_{n+p}(\cM) \to \HC_{n+p+1}(\cM)$, is a contracting chain homotopy.   
\end{proof}

\begin{remark}
\label{rem:auginv}
Let us denote the quasi-isomorphism $\TT(\cM) \to \HC(\cM)$ of the Proposition by $\epsilon$.  Thus,
\[
\epsilon(x_0,\ldots,x_n) = d_n(x_n).
\]
Suppose $z \in \HC_n(\cM)$ is a Hochschild cycle, i.e. $d_0(z)+ \cdots + d_n(z) = 0$.  Then the element 
\[
(s_0(z),s_1(z),\ldots,s_n(z)) \in \bigoplus_{p = 0}^n \HC_{n+1} = \TT_n(\cM)
\]
is a cycle in $\TT$ that maps to $z$ under $\epsilon$.
\end{remark}

The chain complex $\TT(\cM)$ has a $\ZZ/2$-action.  We will denote the generator of this action by $\tau$.  Namely, if $x = (x_0,x_1,\ldots,x_n) \in \TT(\cM)$ then we define 
\[
\tau(x) = (t_{n+1} x_n, t^2_{n+1} x_{n-1},\ldots,t^{j+1}_{n+1}(x_{n-j}),\ldots,t^{n+1}_{n+1}(x_0)).
\]

Since $t_{n+1} \circ \stackrel{n+2}{\cdots} \circ t_{n+1} (x) = x$ and $p+q = n$, we have $\tau^2(x) = x$.  We may therefore form the first quadrant bicomplex

\[
\TT^{\ZZ/2} := 
\left(
\xymatrix{
0 & \TT_*(\cM)\ar[l] \ar@(ul,ur)^{\bdy_{\TT}} & \TT_*(\cM)\ar[l]_{1+\tau}\ar@(ul,ur)^{\bdy_{\TT}}
& \TT_*(\cM)\ar[l]_{1+\tau} \ar@(ul,ur)^{\bdy_{\TT}}&\dots\ar[l]_(.3){1+\tau}
}
\right)
\]
and its periodic version
\[
\TT^{\Tate} := 
\left(
\xymatrix{
\cdots & \TT_*(\cM)\ar[l]_{1+\tau \quad} \ar@(ul,ur)^{\bdy_{\TT}} & \TT_*(\cM)\ar[l]_{1+\tau}\ar@(ul,ur)^{\bdy_{\TT}}
& \TT_*(\cM)\ar[l]_{1+\tau} \ar@(ul,ur)^{\bdy_{\TT}}&\dots.\ar[l]_(.3){1+\tau}
}
\right).
\]

\begin{proposition}
\label{prop:h2dr} Let $\cM$ be a bounded cyclic module, and
suppose that the Hodge-to-de Rham spectral sequence for $\cM$
degenerates at the first page.  Then the spectral sequence $\vhE$
attached to each of the bicomplexes $\TT^{\ZZ/2}$ and $\TT^{\Tate}$
also degenerates at the first page.
\end{proposition}

\begin{proof}
By Proposition \ref{prop:cyclicformality}, we may assume that there is an integer $n$ such that $\HH_i(\cM) = 0$ for $i \neq n$.  By Proposition \ref{prop:TT}, the homology groups $H_i(\TT(\cM))$ also vanish for $i \neq n$.  But then the spectral sequences attached to $\TT^{\ZZ/2}$ and to $\TT^{\Tate}$ have
\[
\vhE^1_{pq} = 0 \text{ for $q \neq n$}
\]
and they therefore collapse.
\end{proof}

\begin{remark}
\label{rem:cyclic}
Suppose $\cM$ is the cyclic module coming from an $\Field$-algebra
$A$~\cite[Proposition 2.5.4]{LodayBook}. In the definition of
$\HC(A)$ from Section~\ref{sec:hoch-background}, if we take $R$ to
be the bar complex $\Barop(A)$ of $A$~\cite[Section 1.1.11]{LodayBook} then
$\HC(\cM)=\HC(A)$.
Moreover, $\TT(\cM)$
is naturally identified with $\HC(\Barop(A))$, i.e., with
$(R_*\otimes_\Field M_*)/\sim$, where $R_*=M_*=\Barop(A)$ and $\sim$
is as in Definition~\ref{def:Hochschild}.
This identification respects the $\ZZ/2$ actions, so 
the spectral sequence $\vhE^r_{pq}$ attached to $\TT^\Tate(\cM)$ agrees
with the spectral sequence $\vhE^r_{pq}$ attached to $\HC^\Tate(A \DTP A)$ for $r \geq 1$.
\end{remark}

\subsubsection{Hodge-to-de Rham formality implies \texorpdfstring{$\pi$}{pi}-formality for Calabi-Yau algebras}
\label{subsubsec:caligula}
In this section, we treat algebras rather than \dg algebras for
simplicity, and for easy reference to~\cite{LodayBook}.

\begin{theorem}
Let $A$ be a finite-dimensional algebra over $\Field$ (regarded as a dg algebra with trivial differential), satisfying the following conditions:
\begin{enumerate}
\item $A$ is homologically smooth.  
\item The Hodge-to-de Rham spectral sequence for $A$ degenerates at $E^1$.  
\item For some integer $d$, there is a quasi-isomorphism of bimodules $\Sigma^d A \cong A^!$.  In other words, $A$ is \emph{Calabi-Yau}.
\end{enumerate}
Then the algebra $A$ is $\pi$-formal.
\end{theorem}

\begin{proof}
  Since condition (3) states that the cobar bimodule $A^!$ is
  quasi-isomorphic to a shift of the diagonal bimodule $A$, it will
  suffice to show that conditions (1) and (2) imply that the diagonal
  bimodule is $\pi$-formal.

By Remark \ref{rem:cyclic}, the Hochschild-Tate spectral sequence of $A \DTP A$ coincides with the $\vhE$ spectral sequence attached to $\TT^\Tate(A^{\otimes (\bullet + 1)})$, and by Proposition \ref{prop:h2dr} if condition (2) holds then this spectral sequence collapses at the first page.  Thus $
\vhE^1_{p,q} = \HH_q(A,A \DTP A)$ degenerates: $\vhE^1_{p,q} = \vhE^\infty_{p,q}$.  Since $A \DTP A \cong A$, we in particular have the equation
\begin{align*}
\sum_{p+q = n} \dim_\Field \vhE^\infty_{p,q} &= \sum_{p+ q = n}
\dim_{\Field} \HH_q(A,A).\\
\shortintertext{We claim that if $A$ is homologically smooth then}
\sum_{p+q = n} \dim_\Field\hvE^\infty_{p,q} &= \sum_{p+q = n}
\dim_{\Field} \HH_q(A,A)\qquad \text{and}\\
\sum_{p+q = n} \dim_\Field\hvE^3_{p,q} &= \sum_{p+q = n} \dim_\Field
\HH_q(A,A).\\
\shortintertext{In particular $\hvE^3 = \hvE^\infty$ so the diagonal bimodule is $\pi$-formal.  The first part of the claim holds because if $A$ is homologically smooth then the Hochschild-Tate bicomplex is acyclic outside of a bounded horizontal strip, so that we also have}
\sum_{p+q = n} \dim_\Field \hvE^\infty_{p,q} &= \sum_{p+q = n} \dim_\Field\vhE^{\infty}_{p,q}.
\end{align*}
The second part of the claim is a consequence of Proposition
\ref{prop:TateE2}.  This completes the proof.
\end{proof}

\begin{remark}
\label{rem:dontknow}
We do not know whether the converse to this theorem holds --- that is, we do not know whether the $\pi$-formality of $A$ implies the degeneration of the Hochschild-to-cyclic spectral sequence for $A$.
\end{remark}

\section{Applications to Heegaard Floer homology}\label{sec:HF-applications}
This section contains the topological applications of the paper. We
start with a selective review of bordered Heegaard Floer homology in
Section~\ref{sec:bordered-background}. In
Section~\ref{sec:loc-bord-cobar} we prove that certain of the bordered
algebras are $\pi$-formal. Using these results,
Section~\ref{sec:br-d-cov} proves Theorems~\ref{thm:br-g2}
and~\ref{thm:br-extreme}. The model for these proofs is
Theorem~\ref{thm:br-d-cov-sseq}, where we show that $\pi$-formality of
the bordered algebras implies Hendricks's localization result
(Theorem~\ref{thm:hendricks}). (The reader may want to skip directly
to Theorem~\ref{thm:br-d-cov-sseq}, to understand the structure of
this argument, and refer back to
Sections~\ref{sec:bordered-background} and~\ref{sec:loc-bord-cobar} as
needed.) Sections~\ref{sec:TC} and~\ref{subsec:double-cover} are
devoted to proving Theorem~\ref{thm:honest-dcov}. In
Section~\ref{sec:TC} we explain how to obtain $\HFa(Y)$ as the
Hochschild homology of a bimodule (if $b_1(Y)>0$) and prove that these
bimodules are neutral (in the sense of
Definition~\ref{def:neutral}). Theorem~\ref{thm:honest-dcov} follows
easily, as is shown in Section~\ref{subsec:double-cover}.

Throughout this section, Heegaard Floer homology groups will have
coefficients in $\Field$.

\subsection{Background on Bordered Floer
  homology}\label{sec:bordered-background}
Bordered (Heegaard) Floer homology is an extension of the Heegaard
Floer $3$-manifold invariant $\HFa(Y)$ to $3$-manifolds with
boundary. It, and Zarev's further extension, bordered-sutured
Floer homology, will allow us to apply Theorem~\ref{thm:hoch-local} to
Heegaard Floer theory. In this section, we briefly review the relevant
aspects of these theories; for more details the reader is referred to
\cite{LOT1,LOT2,Zarev09:BorSut}.

\subsubsection{The algebra associated to a surface}
A \emph{strongly based surface} is a closed, connected, oriented
surface $F$, together with a distinguished disk $D\subset F$.
Morally, bordered Floer homology associates to a strongly based
surface $(F,D)$ a \dg algebra $\Alg(F)$. More precisely, bordered
Floer theory associates a \dg algebra $\Alg(\PMC)$ to a combinatorial
representation $\PMC$ for $(F,D)$ called a \emph{pointed matched
  circle}. We will write $F(\PMC)$ for the strongly based surface
associated to a pointed matched circle $\PMC$.

We will not need the explicit form of the algebra $\Alg(\PMC)$ (except
briefly in the proof of Proposition~\ref{prop:Alg-is-hom-smooth} and,
in a special case, in Section~\ref{sec:loc-bord-cobar}); but three
points will be relevant below. First, if $\PMC$ represents $S^2$
(there is a unique such pointed matched circle) then
$\Alg(\PMC)=\Field$. Second, the algebra $\Alg(\PMC)$ decomposes as a
direct sum: if $F(\PMC)$ has genus $k$ then
\[
\Alg(\PMC)=\bigoplus_{i=-k}^k \Alg(\PMC,i);
\]
the integer $i$ corresponds to a choice of $\SpinC$-structure on $F$.
Third, the bordered algebras are homologically smooth (see
Definition~\ref{def:homol-smooth}):
\begin{proposition}\label{prop:Alg-is-hom-smooth}
  For any pointed matched circle $\PMC$ and integer $i$, the algebra
  $\Alg(\PMC,i)$ is homologically smooth and proper.
\end{proposition}
\begin{proof}
  It is obvious that $\Alg(\PMC,i)$ is homologically proper, since the algebra
  $\Alg(\PMC,i)$ is itself finite-dimensional. The fact that it is
  homologically smooth
  follows from~\cite[Proposition 5.13]{LOT11:HomPair}. Fix a pointed
  matched circle $\PMC$ and let $\Ground$ be the subalgebra of
  idempotents in $\Alg(\PMC,i)$. Let
  $\bAlg=\Hom_\Ground(\Alg(\PMC,-i),\Ground)$ and let 
  \[
  M=\Alg(\PMC,i)\otimes_\Ground\bAlg\otimes_\Ground\Alg(\PMC,i).
  \]
  View $M$ as an $(\Alg(\PMC,i),\Alg(\PMC,i))$-bimodule in the obvious
  way. 
  Let $\Chord(\PMC)$ denote the set of connected chords in
  $\PMC$. Given a chord $\xi\in\Chord(\PMC)$ there is an associated
  algebra element $a(\xi)\in\Alg(\PMC)$.
  Endow $M$ with a differential defined by
  \begin{multline*}
    d(x\otimes\phi\otimes y)= \sum_{\xi\in\Chord(\PMC)} \bigl(x\cdot
    a(\xi)\bigr)\otimes \bigl(a(\xi)\cdot \phi\bigr)\otimes y
    +\sum_{\xi\in\Chord(\PMC)} x\otimes \bigl(\phi\cdot a(\xi)\bigr)
    \otimes \bigl(a(\xi)\cdot y\bigr)\\
    +d(x)\otimes\phi\otimes y+x\otimes \bar{d}(\phi) \otimes
    y+x\otimes\phi\otimes d(y).
  \end{multline*}
  (The module $M$ is the modulification of the type \DD\ structure
  $\SmallBar$ from \cite[Section 5.4]{LOT11:HomPair}.)

  It follows from \cite[Proposition 5.13]{LOT11:HomPair} that $M$ is
  quasi-isomorphic to $\Alg(\PMC,i)$. It remains to verify that $M$ is a
  finite cell retract. Let 
  \[
  N=\Alg(\PMC,i)\otimes_\Field\bAlg\otimes_\Field\Alg(\PMC,i),
  \]
  with differential defined by the same formula as the differential on
  $M$. 

  We verify that $M$ is a retract of $N$.  Let $\{a_i\}$ be the
  standard basis for $\Alg(\PMC,i)$, and let $\{a_j^*\}$ be the dual
  basis for $\bAlg$. Each $a_i$ has a left idempotent and a right
  idempotent, i.e., indecomposable idempotents $I$ and $J$
  (respectively) so that $I\cdot a_i\cdot J=a_i$.  Call an element
  $a_i\otimes_\Field a_j^*\otimes_\Field a_k$ of $N$ \emph{consistent}
  if the right idempotent of $a_i$ is the same as the left idempotent
  of $a_j^*$ and the right idempotent of $a_j^*$ is the same as the
  left idempotent of $a_k$.  The span (over $\Field$) of the set of
  consistent elements of $N$ is a submodule of $N$, and is isomorphic
  to $M$. There is an obvious retraction $r\co N\to M$ which sends any
  inconsistent basic element to zero; equivalently, $r$ is defined by
  \[
  r(x\otimes_\Field \phi\otimes_\Field y)=x\otimes_\Ground
  \phi\otimes_\Ground y.
  \]

  Finally, we verify that $N$ is a finite cell bimodule. Recall that
  each basic algebra element $a_i$ of $\Alg(\PMC,i)$ has a support
  $\supp(a_i)$ in $(\ZZ_{\geq 0})^{4k-1}$.  Note that if
  $a(\xi)a_i=a_j$ or $a_ia(\xi)=a_j$ for some nontrivial chord $\xi$
  then $\supp(a_i)<\supp(a_j)$. Consequently, if $a(\xi)a_i^*=a_j^*$
  or $a_i^*a(\xi)=a_j^*$ for some nontrivial chord $\xi$ then
  $\supp(a_i)>\supp(a_j)$.

  Define a partial order on $\{a_i\}$ by declaring that $a_i<a_j$ if
  either:
  \begin{itemize}
  \item $\supp(a_i)<\supp(a_j)$ or
  \item $\supp(a_i)=\supp(a_j)$ and $a_i$ has more crossings then
    $a_j$.
  \end{itemize}
  There is a corresponding partial order on $\bAlg$ defined by
  $a_i^*<a_j^*$ if and only if $a_i<a_j$. From the observations of
  the previous paragraph, it is immediate that:
  \begin{itemize}
  \item If $a(\xi)a_i^*=a_j^*$ or $a_i^*a(\xi)=a_j^*$ then $a_i^*>a_j^*$.
  \item If $\overline{d}(a_i^*)=a_j^*$ then $a_i^*>a_j^*$.
  \end{itemize}
  Choose a total ordering of the $a_i$ compatible with the partial
  ordering $<$; re-indexing, we may assume this ordering is
  $a_1,a_2,\dots,a_\ell$. Let $N_n$ be the sub-bimodule of $N$ generated
  by $a_1,\dots,a_n$. It follows that $d(N_n)\subset N_n$;
  $N_{n-1}\subset N_{n}$; and $N_n/N_{n-1}=\Alg(\PMC,i)\otimes_\Field
  a_n\otimes_\Field\Alg(\PMC,i)$. Thus, the sequence of submodules
  $0\subset N_1\subset N_2\subset\dots\subset N_\ell=N$ present $N$ as a
  finite cell bimodule. The result follows.
\end{proof}

\begin{remark}
  It is not hard to show that the modulification of any
  finite-dimensional, bounded type \DD\ bimodule is a finite cell
  retract.
\end{remark}

\subsubsection{Bimodules associated to \texorpdfstring{$3$}{3}-dimensional cobordisms}
By an \emph{arced cobordism from $F(\PMC_1)$ to $F(\PMC_2)$ } we mean a
$3$-dimensional cobordism $Y$ from $F(\PMC_1)$ to $F(\PMC_2)$ together
with a framed arc (or $[0,1]\times \bD^2$) connecting the
distinguished disks in $F(\PMC_1)$ and $F(\PMC_2)$.  Bordered Floer
homology associates an $\Ainf$ $(\Alg(\PMC_1),\Alg(\PMC_2))$ bimodule
$\CFDAa(Y)$ to an arced cobordism from $F(\PMC_1)$ to $F(\PMC_2)$. As
with the algebra, the definition of $\CFDAa(Y)$ will be largely
unimportant for us; but we will need the following properties of it.
\begin{enumerate}
\item In the case that both boundary components of $Y$ are copies of
  $S^2$, $\CFDAa(Y)$, which is a bimodule over $\Alg(S^2)=\Field$, is
  quasi-isomorphic to $\CFa(Y\cup_\bdy(B^3\amalg B^3))$, the chain
  complex computing the (ordinary, closed) Heegaard Floer invariant $\HFa$ of
  the $3$-manifold obtained by capping off the boundary components of
  $Y$.
\item The invariant $\CFDAa(Y)$ is not associated directly to $Y$, but
  rather to a combinatorial representation for $Y$ called an
  \emph{arced, bordered Heegaard diagram} (see~\cite[Definition
  5.4]{LOT2}). $\CFDAa(Y)$ is an $\Ainf$ bimodule, and is well-defined
  up to $\Ainf$ homotopy equivalence~\cite[Theorem 10]{LOT2}.
\item Although $\CFDAa(Y)$ is an $\Ainf$-bimodule, it is $\Ainf$
  homotopy equivalent to an honest \dg bimodule. (This can be proved
  either topologically or algebraically. For the topological proof,
  one can choose a Heegaard diagram for $Y$ so that computing $\CFDAa$
  with respect to this diagram gives an honest \dg bimodule;
  compare~\cite[Chapter 8]{LOT1}. The algebraic proof holds for
  $\Ainf$ bimodules quite generally; see, for instance,~\cite[Section
  2.4.1]{LOT2}.)

  In particular, this point allows us to apply
  Theorem~\ref{thm:hoch-local}, which was proved in the context of \dg
  modules, to $\CFDAa(Y)$.
\item Gluing $3$-dimensional cobordisms corresponds to tensoring
  bimodules:
  \begin{citethm}\label{thm:pairing}\cite[Theorem 12]{LOT2}
    Let $Y_{12}$ be an arced cobordism from $F(\PMC_1)$ to $F(\PMC_2)$
    and $Y_{23}$ an arced cobordism from $F(\PMC_2)$ to
    $F(\PMC_3)$. Then
    \[
    \CFDAa(Y_1\cup_{F(\PMC_2)}Y_2)\simeq
    \CFDAa(Y_1)\DTP_{\Alg(\PMC_2)}\CFDAa(Y_2).
    \]
  \end{citethm}
\item Roughly, self-gluing a $3$-dimensional cobordism corresponds to
  Hochschild homology. More accurately, when one self-glues an arced
  cobordism, the arc gives rise to a knot, and the Hochschild homology
  takes this knot into account:
  \begin{citethm}\label{thm:bordered-hochschlid}\cite[Theorem
    14]{LOT2} Let $Y$ be an arced cobordism from $F(\PMC)$ to itself. Let
    $Y_0$ be the result of gluing the two boundary components of $Y$
    together (via the identity map) and let $\gamma$ be the framed knot in
    $Y_0$ coming from the arc in $Y$.  Let $(Y^\circ,K)$ be the open
    book obtained by performing surgery on $\gamma\subset
    Y_0$. Then 
    \[
    \HFKa(Y^\circ,K)\cong\HH_*(\CFDAa(Y)).
    \]
  \end{citethm}
\item The grading on $\CFDAa(Y)$ is fairly subtle: it is graded by a
  $G$-set, where $G$ is a non-commutative group. Therefore, the
  Hochschild complex $\HC_*(\CFDAa(Y))$ is not necessarily
  $\ZZ$-graded. To apply Theorem~\ref{thm:hoch-local}, we must
  restrict to cases in which the Hochschild complex is $\ZZ$-graded.
\end{enumerate}

\subsubsection{The bordered-sutured setting}
In \cite{Zarev09:BorSut}, Zarev put bordered Floer homology in a more
general framework, called \emph{bordered sutured Floer homology}. As we will
use this setting below, we recall it now.

\begin{definition}\cite[Definition 1.2]{Zarev09:BorSut}
  A \emph{sutured surface} is a tuple $(F,S_+,S_-)$ where $F$ is a
  surface with boundary and $S_+,S_-$ are codimension-$0$ submanifolds
  of $\bdy F$ so that $S_+\cap S_-=\bdy S_+=\bdy S_-$ and $S_+\cup
  S_-=\bdy F$. We write $\Gamma$ for $S_+\cap S_-$. We require that
  $S_+$ and $S_-$ have no closed components (i.e., circles) and that
  $F$ have no closed components (i.e., closed sub-surfaces).
\end{definition}

There are combinatorial representations, called \emph{arc diagrams},
for sutured surfaces; this is a generalization of the notion of a
pointed matched circle. Given an arc diagram $\PMC$ we write
$\PunctF(\PMC)=(\PunctF(\PMC),S_+(\PMC),S_-(\PMC))$ for the associated
sutured surface.

Pointed matched circles are special cases of arc diagrams. 

\begin{example}
  Given a pointed matched circle $\PMC$, let $D$ denote the
  distinguished disk in $F(\PMC)$. Then
  $\PunctF(\PMC)=F(\PMC)\setminus \interior(D)$.  $S_+(\PMC)$ and
  $S_-(\PMC)$ are connected arcs in $\bdy D$ intersecting at their
  endpoints.
\end{example}

Associated to any arc diagram $\PMC$ is a \dg algebra $\Alg(\PMC)$.
In the special case that $\PMC$ is a pointed matched circle the
bordered Floer algebra $\Alg(\PMC)$ and the bordered-sutured Floer
algebra $\Alg(\PMC)$ are the same.

\begin{definition}\cite[Definition 1.3]{Zarev09:BorSut}
  A \emph{$3$-dimensional sutured cobordism} from $\PunctF(\PMC_L)$ to
  $\PunctF(\PMC_R)$ consists of the following
  data:
  \begin{itemize}
  \item A $3$-manifold with boundary $Y$.
  \item Codimension-$0$ subsets $R_\pm\subset \bdy Y$.
  \item A homeomorphism 
    \[
    (\phi_L\amalg \phi_R)\co
    \bigl(-\PunctF(\PMC_L)\amalg \PunctF(\PMC_R)\bigr)\to
    Y\setminus\bigl(\interior(R_+\cup R_-)\bigr).
    \]
  \end{itemize}
  These data are required to satisfy the following properties:
  \begin{itemize}
  \item $\phi_L(S_+(\PMC_L))\subset R_+$, $\phi_L(S_-(\PMC_L))\subset
    R_-$ $\phi_R(S_+(\PMC_R))\subset R_+$ and
    $\phi_R(S_-(\PMC_R))\subset R_-$.
  \item Neither $R_+$ nor $R_-$ has any closed components.
  \end{itemize}
\end{definition}

Given a sutured cobordism $(Y,R_\pm,\phi_L,\phi_R)$, let $\Gamma$
denote the one-manifold with boundary $R_+\cap R_-$. The curves in
$\Gamma$ are called \emph{sutures}. Orient $\Gamma$ as the boundary of
$R_+$. Then we can reconstruct $R_\pm$ from $\Gamma$ (and vice-versa).

\begin{example}\label{eg:sutured-from-arced}
  Let $Y$ be a $3$-dimensional arced cobordism from $F(\PMC_1)$ to
  $F(\PMC_2)$, with arc $\gamma$. Then $Y\setminus \nbd(\gamma)$ is
  naturally a sutured cobordism as follows. The identification of
  $(-F(\PMC_1))\amalg F(\PMC_2)$ with $\bdy Y$ induces an identification
  of $(-\PunctF(\PMC_1))\amalg \PunctF(\PMC_2)$ with $(\bdy
  Y)\setminus(\nbd(\bdy\gamma))$. Write $\bdy\nbd(\gamma)\cong
  \bD^2\cup[0,1]\times S^1\cup\bD^2$. Regarding $S_\pm(\PMC_i)$ as
  subsets of $\bdy\bD^2=\bdy \PunctF(\PMC_i)$, we may choose the
  identification of $\bdy\nbd(\gamma)$ in such a way that
  $S_+(\PMC_1)$ and $S_+(\PMC_2)$ are the same subset of $\bdy\bD^2$
  (and so $S_-(\PMC_1)$ and $S_-(\PMC_2)$ are also the same subset of
  $\bdy\bD^2$).  Then $R_\pm$ is given by $[0,1]\times S_\pm$.
\end{example}

To each $3$-dimensional sutured cobordism $Y$ from
$\PunctF(\PMC_L)$ to $\PunctF(\PMC_R)$ Zarev associates an
$(\Alg(\PMC_L),\Alg(\PMC_R))$-bimodule $\BSDAa(Y)$.

\begin{example}\label{eg:bsda-of-arced}
  If $Y$ is an arced cobordism and $Y'$ is the associated sutured
  cobordism (see Example~\ref{eg:sutured-from-arced}) then
  $\BSDAa(Y')\cong \CFDAa(Y)$.
\end{example}

\begin{example}
  If $Y$ is a sutured cobordism from $\emptyset$ to $\emptyset$ then
  $Y$ is an ordinary sutured manifold. If moreover
  $\chi(R_+)=\chi(R_-)$ (i.e., $Y$ is \emph{balanced}) then
  $\BSDAa(Y)\cong \SFH(Y)$, Juh\'asz's \emph{sutured Floer homology}
  (see \cite{Juhasz06:Sutured}).
\end{example}

These bimodules satisfy a pairing theorem, analogous to
Theorem~\ref{thm:pairing}:
\begin{citethm}\label{thm:sutured-pairing}\cite[Theorem 8.7]{Zarev09:BorSut}
  Let $Y_{12}$ be a sutured cobordism from $\PunctF(\PMC_1)$ to $\PunctF(\PMC_2)$
  and $Y_{23}$ a sutured cobordism from $\PunctF(\PMC_2)$ to
  $\PunctF(\PMC_3)$. Then
  \[
  \BSDAa(Y_1\cup_{\PunctF(\PMC_2)}Y_2)\simeq
  \BSDAa(Y_1)\DTP_{\Alg(\PMC_2)}\BSDAa(Y_2).
  \]
\end{citethm}

The self-gluing theorem is conceptually clearer in this language. Let
$(Y,R_\pm)$ be a sutured cobordism from $\PunctF(\PMC)$ to itself. Assume
that $\chi(R_+)=\chi(R_-)$. Let $Y^\circ$ be the result of gluing the
two boundary components of $Y$ together (via the identity map) and
$R_\pm^\circ$ the image of $R_\pm$ in $Y^\circ$. Then
$(Y^\circ,R_\pm^\circ)$ is a balanced sutured manifold; the balanced
condition comes from the condition on the Euler characteristic of
$R_\pm$.
\begin{citethm}\label{thm:sutured-self-pairing}
  With notation as above, the sutured Floer homology of
  $(Y^\circ,R_\pm^\circ)$ is given by
  \[
  \SFH(Y^\circ)\cong \HH_*(\BSDAa(Y)).
  \]
\end{citethm}
\begin{proof}
  Let $\Id$ denote the identity sutured cobordism from $\PunctF(\PMC)$ to
  itself. Then $\BSDAa(\Id)$ is the $(\Alg(\PMC),\Alg(\PMC))$-bimodule
  $\Alg(\PMC)$. Let $\BSAa(\Id)$ denote the bordered-sutured invariant
  of $\Id$ viewed as a cobordism from $\emptyset$ to $\PunctF(-\PMC)\amalg
  \PunctF(\PMC)$ and let $\BSDa(Y)$ denote the bordered-sutured invariant of
  $Y$ viewed as a cobordism from $\PunctF(-\PMC)\amalg \PunctF(\PMC)$ to
  $\emptyset$. Recall that $\Alg(-\PMC)=\Alg(\PMC)^\op$ and
  $\Alg(\PMC_1\amalg\PMC_2)=\Alg(\PMC_1)\otimes_\Field\Alg(\PMC_2)$.
  We have
  \begin{align*}
    \HH_*(\BSDAa(Y))&\cong
    H_*(\BSDAa(\Id)\DTP_{\Alg(\PMC)\otimes\Alg(\PMC)^\op}\BSDAa(Y))\\
    &\cong
    H_*(\BSAAa(\Id)\DTP_{\Alg(\PMC)\otimes\Alg(-\PMC)}(\BSDDa(\Id)\otimes \BSDAa(\Id)) \DTP_{\Alg(\PMC)\otimes\Alg(\PMC)^\op} \BSDAa(Y))\\
    &\cong
    H_*\bigl(\BSAa(\Id)\DTP_{\Alg(\PMC\amalg (-\PMC))}\BSDa(Y)\bigr)\\
    &\cong 
    \SFH(\Id\cup_\bdy Y)=\SFH(Y^\circ).
  \end{align*}
  Here, the first isomorphism is the definition of Hochschild homology.
  The remaining isomorphisms use Theorem~\ref{thm:sutured-pairing};
  the second also uses 
  the fact that in bordered-sutured Floer homology, disjoint
  union corresponds to tensor product over $\Field$, and the third
  uses the fact that $\BSAa(M)$ is simply $\BSAAa(M)$ viewed as a
  module over $\Alg(\PMC\amalg (-\PMC))$.
\end{proof}

\begin{example}
  Suppose that $Y$ is an arced cobordism inducing a sutured manifold
  $Y'$ as in Example~\ref{eg:sutured-from-arced}. Then
  $\SFH((Y')^\circ)\cong \HFKa(Y^\circ,K)$, and $\BSDAa(Y')\cong
  \CFDAa(Y)$ (Example~\ref{eg:bsda-of-arced}), so
  Theorem~\ref{thm:sutured-self-pairing} recovers
  Theorem~\ref{thm:bordered-hochschlid}.
\end{example}

\begin{proposition}\label{prop:Alg-is-hom-smooth-borsut}
  For any arc diagram $\PMC$ the algebra $\Alg(\PMC)$ is homologically
  smooth.
\end{proposition}
\begin{proof}
  The proof is the same as the proof of
  Proposition~\ref{prop:Alg-is-hom-smooth}.
\end{proof}

\subsection{Localization for the cobar complex}\label{sec:loc-bord-cobar}
In order to obtain localization results, we will use special cases of
the following:
\begin{conjecture}\label{conj:HF-pi-formal}
  For any arc diagram $\PMC$ and integer $i$, the algebra
  $\Alg(\PMC,i)$ is $\pi$-formal (Definition~\ref{def:pi-formal}).
\end{conjecture}

Any case of Conjecture~\ref{conj:HF-pi-formal} gives a family of
localization results. Note that this conjecture is entirely
combinatorial. Since $\Alg(\PMC,i)$ is homologically smooth
(Proposition~\ref{prop:Alg-is-hom-smooth-borsut}), verifying the
conjecture in any particular case is a finite problem.

\begin{figure}
  \centering
  \includegraphics{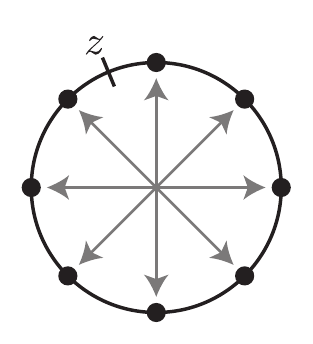}
  \caption{\textbf{The antipodal pointed matched circle.} The genus
    $2$ case is shown; the matching is indicated with gray arrows. See
    also~\cite[Example 3.20]{LOT1}.}
  \label{fig:antipodal-pmc}
\end{figure}

We will prove two special cases of Conjecture~\ref{conj:HF-pi-formal}:
\begin{theorem}
  \label{thm:HF-pi-formal-extreme} Let $\PMC$ be the antipodal pointed
  matched circle (Figure~\ref{fig:antipodal-pmc}) for a surface of
  genus $k$. Then $\Alg(\PMC,-k+1)$ is $\pi$-formal.
\end{theorem}

\begin{theorem}
  \label{thm:HF-pi-formal-small} Let $\PMC$ be the antipodal pointed
  matched circle for a surface of genus $k\leq 2$. Then for any $i$,
  $\Alg(\PMC,i)$ is $\pi$-formal.
\end{theorem}

We start by proving Theorem~\ref{thm:HF-pi-formal-extreme}, but first
recall some facts about the algebra $\Alg=\Alg(\PMC,-k+1)$. The
differential on $\Alg$ vanishes; and $\Alg$ has a simple description
as a path algebra with relations:
\[
\Alg=\Biggl(
  \xymatrix{ \iota_1\ar@/^/[r]^{a_1}\ar@/_/[r]_{b_1} &
    \iota_2\ar@/^/[r]^{a_2}\ar@/_/[r]_{b_2} &
    \cdots\ar@/^/[r]^{a_{2k-1}}\ar@/_/[r]_{b_{2k-1}} &
    \iota_{2k}\ar@/^3pc/[lll]_c
  }\bigg/a_ib_{i+1}=b_ia_{i+1}=b_{2k-1}c=ca_1=0\Biggr).
\]
The algebra $\Alg$ is quadratic. Its quadratic dual is given by
\[
\Blg=\Biggl(\xymatrix{
\iota_1\ar@/_3pc/[rrr]_\cprime &
\iota_2\ar@/_/[l]_{\aprime_1}\ar@/^/[l]^{\bprime_1} &
\cdots\ar@/_/[l]_{\aprime_{2}}\ar@/^/[l]^{\bprime_{2}} & 
\iota_{2k}
\ar@/_/[l]_{\aprime_{2k-1}}\ar@/^/[l]^{\bprime_{2k-1}}
}\bigg/\aprime_i\aprime_{i+1}=\bprime_i\bprime_{i+1}=\cprime\aprime_{2k-1}=\bprime_1\cprime=0\Biggr).
\]
(In fact, $\Alg$ and $\Blg$ are isomorphic, but it will be clearer to
view them as distinct.)

The following is essentially a special case of results from~\cite{LOT11:HomPair}:
\begin{proposition}
  The algebra $\Alg$ is Koszul (over its subalgebra $\Ground$ of idempotents).
\end{proposition}
\begin{proof}
  Given a pointed matched circle $\PMC$, we can form $-\PMC$, the
  orientation-reverse of $\PMC$. We can also form the dual pointed
  matched circle $\PMC_*$: if we think of $\PMC$ as a handle
  decomposition coming from a Morse function $f\co F(\PMC)\to\RR$ then
  $\PMC_*$ corresponds to $-f$. The algebra $\Blg$ is simply
  $\Alg(\PMC_*,-k+1)$. It is explained in~\cite[Section
  8.2]{LOT11:HomPair} that $\Alg(\PMC,i)$ is Koszul dual (in a
  particular sense) to both $\Alg(-\PMC,-i)$ and $\Alg(\PMC_*,i)$. So,
  the work in proving the present proposition is simply translating
  that result into the language of this paper.

  As in the proof
  of~\cite[Theorem 13]{LOT11:HomPair}, consider the type \DD\ bimodule
  $\lsup{\Alg}\DD(\frac{\Id}{2})^\Blg$ associated to the diagram
  $\mathcal{G}(\PMC)$ of~\cite[Construction
  8.18]{LOT11:HomPair}. By~\cite[Proposition 8.13]{LOT11:HomPair} and the
  proof of~\cite[Theorem 13]{LOT11:HomPair},
  $\lsub{\Alg}\Alg_\Alg\DT\lsup{\Alg}\DD(\frac{\Id}{2})^\Blg\DT \lsub{\Blg}\Blg$ is a resolution
  of $\Ground$. But the bimodule $\lsup{\Alg}\DD(\frac{\Id}{2})^\Blg$
  is computed explicitly in~\cite[Proposition 3.22]{LOT13:faith}; in
  particular, it follows from that description that
  $\lsub{\Alg}\Alg_\Alg\DT\lsup{\Alg}\DD(\frac{\Id}{2})^\Blg\DT \Blg$ is the Koszul
  complex.
\end{proof}

In particular, the Koszul resolution of $\Alg$ is given by
$\Alg\otimes\Blg^*\otimes\Alg$, with differential
\begin{align*}
\bdy (x\otimes f\otimes z)&=\sum_{i=1}^{2k-1} (x a_i\otimes \aprime_i f\otimes z+x
 b_i\otimes \bprime_i f\otimes z+x\otimes f \aprime_i\otimes a_i z+x\otimes
f \bprime_i\otimes b_i  z)\\
&\qquad\qquad+ x  c\otimes \cprime  f\otimes z + x\otimes
f \cprime\otimes c z.
\end{align*}
(Here, in Corollary~\ref{cor:Hoch-coh-small-models}, and in the proof
of Theorem~\ref{thm:HF-pi-formal-extreme}, $\otimes$ means the tensor
product over $\Ground$, the subalgebra of idempotents. In particular,
we are using the identification between the idempotents of $\Alg$ and
$\Blg$ given by the labeling of vertices in the path algebra
description above.)

Using this Koszul resolution, we get a model for $A^!$:
\[
A^! = \Hom_{\Bimod{A}}((\Alg\otimes\Blg^*\otimes\Alg,\bdy),\Alg^e) = (\Alg\otimes\Blg\otimes\Alg,\bdy^T)
\]
(see Section~\ref{sec:hoch-background}), where $\bdy^T$ denotes map
induced by $\bdy$. Using this model, we have:

\begin{corollary}\label{cor:Hoch-coh-small-models}
  The Hochschild homology of $\Alg^!$ is the homology of the chain
  complex $\Alg\otimes\Blg/\sim$, where $x\otimes y\iota \sim \iota
  x\otimes y$ for each idempotent $\iota$, with differential
  \begin{equation}\label{eq:small-ext-A}
  \bdy(x\otimes y)=\sum_{i=1}^{2k-1} (xa_i\otimes \aprime_iy
  +a_ix\otimes y\aprime_i 
  + xb_i\otimes \bprime_iy
  +b_ix\otimes y\bprime_i
  )
  +xc\otimes \cprime y+cx\otimes y\cprime.
  \end{equation}
  
  Similarly, $\HH_*(\Alg^!\DTP_{\Alg}\Alg^!)$ is given by
  $\Alg\otimes\Blg\otimes\Alg\otimes\Blg/\sim$, where $x_1\otimes
  y_1\otimes x_2\otimes y_2\iota \sim \iota x_1\otimes y_1\otimes
  x_2\otimes y_2$ for each idempotent $\iota$, with differential
  \begin{equation}\label{eq:small-ext-bar}
  \begin{split}
  \bdy(x_1\otimes y_1\otimes x_2\otimes
  y_2)&=\hspace{-2em}\sum_{(\xi,\xi')\in
    \{(a_i,\aprime_i),(b_i,\bprime_i),(c,\cprime)\}}\hspace{-2em}
  (x_1\xi\otimes \xi'y_1\otimes x_2\otimes y_2+
  x_1\otimes y_1\xi'\otimes \xi x_2\otimes y_2\\
  &\qquad\qquad\qquad +
  x_1\otimes y_1\otimes x_2\xi \otimes \xi'y_2+
  \xi x_1\otimes y_1\otimes x_2\otimes y_2\xi').
  \end{split}
  \end{equation}
\end{corollary}

\begin{proof}[Proof of Theorem~\ref{thm:HF-pi-formal-extreme}]
  This is a somewhat long, concrete computation. To keep notation
  shorter, we will replace the symbol $\otimes$ with a vertical bar
  $|$. Similarly, let $\ell=2k-1$.

  In the computation, we will frequently use the following phenomenon:
  
  \emph{Vanishing phenomenon.} If $\xi,\eta\in\{a_i,b_i,c\}$ then
  $\xi\eta\neq 0$ implies that $\eta'\xi'=0$. So, $\xi\eta |\eta'\xi'$
  always vanishes, as does $\eta'\xi'|\xi\eta$.

  The element $(1|1)$ in the model for $\HC_*(\Alg^!)$ given in
  Formula~\eqref{eq:small-ext-A} corresponds to the element $1\in
  \Hom(A^!,A^!)$, and so we want to show that the elements
  $d^{2i}(1|1)$ vanish for all $i\geq 2$.  To this end, consider the
  element $(1|1|1|1)$ in the model for $\HC_*(\Alg^!\DTP\Alg^!)$ given
  in Formula~\eqref{eq:small-ext-bar}; note that $(1|1|1|1)$
  corresponds to $(1|1)$ under the isomorphism of
  Proposition~\ref{prop:Frob}. We will compute the differentials in
  the spectral sequence as in Remark~\ref{rem:hv}.
  
  We have
  \begin{align}
    \bdy(1|1|1|1)&= \sum_{\xi\in \{a_i,b_i,c\}}
    (\xi|\xi'|1|1)+(1|\xi'|\xi|1)+(1|1|\xi|\xi')+(\xi|1|1|\xi')\\
    &=(1+\tau)\bigl(\sum_{\xi\in \{a_i,b_i,c\}}(\xi|\xi'|1|1)+(1|\xi'|\xi|1)\bigr)\label{eq:d1}.
  \end{align}
  Let $(1+\tau)^{-1}\eqref{eq:d1}$ denote the result of dropping the
  $(1+\tau)$ from Formula~\eqref{eq:d1}. Then
  \begin{align}
    \bdy\circ (1+\tau)^{-1}\eqref{eq:d1}
    &= \sum_{\xi,\eta\in \{a_i,b_i,c\}}
    (\xi\eta|\eta'\xi'|1|1)+(\xi|\xi'\eta'|\eta|1)+(\xi|\xi'|\eta|\eta')+(\eta\xi|\xi'|1|\eta')
    \nonumber\\
    & \qquad\qquad\qquad\qquad+(\eta|\xi'|\xi|\eta')+(\eta|\eta'\xi'|\xi|1)+(1|\xi'\eta'|\eta\xi|1)+(1|\xi'|\xi\eta|\eta').
    \label{eq:d1bdy}
  \end{align}
  In Expression~\eqref{eq:d1bdy}, the first and seventh terms are
  identically zero, by the vanishing phenomenon above. When summing over $\xi$ and $\eta$, the second and sixth cancel.  The sum over $\xi$ and $\eta$ of the eighth term is equal to tau applied to the sum over $\xi$ and $\eta$ of the fourth term.
   Further:
  \begin{align*}
    \sum_{\xi,\eta}(\xi|\xi'|\eta|\eta')&=\sum_{i=1}^\ell (a_i|\aprime_i|a_i|\aprime_i)+(a_i|\aprime_i|b_i|\bprime_i)+(b_i|\bprime_i|a_i|\aprime_i)+(b_i|\bprime_i|b_i|\bprime_i)
    \\
    \sum_{\xi,\eta}(\eta|\xi'|\xi|\eta')&=\sum_{i=1}^\ell(a_i|\aprime_i|a_i|\aprime_i)+(a_i|\bprime_i|b_i|\aprime_i)+(b_i|\aprime_i|a_i|\bprime_i)+(b_i|\bprime_i|b_i|\bprime_i)\\
    \sum_{\xi,\eta}(\xi|\xi'|\eta|\eta')+(\eta|\xi'|\xi|\eta')&=\sum_{i=1}^\ell(a_i|\aprime_i|b_i|\bprime_i)+(b_i|\bprime_i|a_i|\aprime_i)+(a_i|\bprime_i|b_i|\aprime_i)+(b_i|\aprime_i|a_i|\bprime_i).
  \end{align*}
  (In verifying these equations, keep in mind that we are tensoring over the idempotents.)
  Substituting in, we have:
  \begin{equation}
    \eqref{eq:d1bdy}=(1+\tau)\bigl(\sum_{\xi,\eta}(\eta\xi|\xi'|1|\eta')+\sum_{i=1}^\ell(a_i|\aprime_i|b_i|\bprime_i)+(b_i|\aprime_i|a_i|\bprime_i)\bigr).\label{eq:d2}
\end{equation}
Differentiating again,
\begin{equation}\label{eq:d2bdy}
  \begin{split}
    \bdy \circ (&1+\tau)^{-1}\eqref{eq:d2}\\
    &=\sum_{\eta,\xi,\nu\in\{a_i,b_i,c\}}(\eta\xi\nu|\nu'\xi'|1|\eta')+(\eta\xi|\xi'\nu'|\nu|\eta')+(\eta\xi|\xi'|\nu|\nu'\eta')+(\nu\eta\xi|\xi'|1|\eta'\nu')\\
    &+\sum_{\nu\in\{a_i,b_i,c\}}\sum_{i=1}^\ell
    (a_i|\aprime_i\nu|\nu b_i|\bprime_i)+(\nu
    a_i|\aprime_i|b_i|\bprime_i\nu) + (a_i\nu|\nu\bprime_i|b_i|\aprime_i)+(\aprime_i|\bprime_i|b_i\nu|\nu\aprime_i).
    \end{split}
  \end{equation}
  Here, we have omitted some terms from the second sum which are zero
  according to the vanishing principle above (e.g.,
  $(a_i\nu|\nu\aprime_i|b_i|\bprime_i)$). In
  Formula~\eqref{eq:d2bdy}, the first and fourth terms vanish
  identically, by the vanishing principle. Next:
  \begin{align*}
    \sum_{\eta,\xi,\nu} (\eta\xi|\xi'\nu'|\nu|\eta')&=(a_{\ell}c|\cprime\bprime_{\ell}|b_{\ell}|\aprime_{\ell})+\sum_{i=1}^{\ell-1} (a_ia_{i+1}|\aprime_{i+1}\bprime_i|b_i|\aprime_i)+(b_ib_{i+1}|\bprime_{i+1}\aprime_i|a_i|\bprime_i)\\
    \sum_{\eta,\xi,\nu}
    (\eta\xi|\xi'|\nu|\nu'\eta')&=(cb_1|\bprime_1|a_1|\aprime_1\cprime)+\sum_{i=1}^{\ell-1}(a_ia_{i+1}|\aprime_{i+1}|b_{i+1}|\bprime_{i+1}\aprime_i)+(b_ib_{i+1}|\bprime_{i+1}|a_{i+1}|\aprime_{i+1}\bprime_i)    
  \end{align*}
  and
  \begin{align*}
    \sum_{\nu}\sum_{i=1}^\ell (a_i|\aprime_i\nu|\nu b_i|\bprime_i)&+(\nu
    a_i|\aprime_i|b_i|\bprime_i\nu)\\
    &\quad= (a_1|\aprime_1\cprime|cb_1|\bprime_1)+\sum_{i=1}^{\ell-1}(a_{i+1}|\aprime_{i+1}\bprime_{i}|b_{i}b_{i+1}|\bprime_{i+1})+(a_{i}a_{i+1}|\aprime_{i+1}|b_{i+1}|\bprime_{i+1}\aprime_{i})\\
    \sum_{\nu}\sum_{i=1}^{\ell}(a_i\nu|\nu\bprime_i|b_i|\aprime_i)&+(\aprime_i|\bprime_i|b_i\nu|\nu\aprime_i)\\
    &\quad=(a_{\ell}c|\cprime\bprime_{\ell}|b_{\ell}|\aprime_{\ell})+\sum_{i=1}^{\ell-1}(a_ia_{i+1}|\aprime_{i+1}\bprime_i|b_i|\aprime_i)+(a_i|\bprime_i|b_ib_{i+1}|\bprime_{i+1}\aprime_i).
  \end{align*}
  So, 
  \begin{equation}
    \label{eq:d2bdy-v2}
    \eqref{eq:d2bdy}=(1+\tau)\bigl((a_1|\aprime_1\cprime|cb_1|\bprime_1)+\sum_{i=1}^{\ell-1}(a_{i+1}|\aprime_{i+1}\bprime_{i}|b_{i}b_{i+1}|\bprime_{i+1})+(a_i|\bprime_i|b_ib_{i+1}|\bprime_{i+1}\aprime_i)\bigr).
    \end{equation}
    Finally,
    \begin{align*}
    \bdy\circ
    (1+\tau)^{-1}\eqref{eq:d2bdy-v2}&=\sum_{i=1}^{\ell-1}(a_{i}a_{i+1}|\aprime_{i+1}\bprime_{i}|b_{i}b_{i+1}|\bprime_{i+1}\aprime_{i})+(a_ia_{i+1}|\aprime_{i+1}\bprime_i|b_ib_{i+1}|\bprime_{i+1}\aprime_i)\nonumber\\
    &=0.\nonumber\qedhere
  \end{align*}
\end{proof}

\begin{proof}[Proof of Theorem~\ref{thm:HF-pi-formal-small}]
  The cases $k=0$, $k=1, i\neq 0$, and $k=1,i\not\in\{-1,0,1\}$ are
  trivial (the algebras are either $0$ or $\Field$). The cases
  $(k,i)=(1,0)$ and $(2,-1)$ follow from
  Theorem~\ref{thm:HF-pi-formal-small}. The case $(k,i)=(2,1)$
  follows from Theorem~\ref{thm:HF-pi-formal-small} and the fact that
  $\Alg(\PMC,i)$ is quasi-isomorphic to $\Alg(\PMC_*,-i)$, which in
  turn is a special case of~\cite[Theorem 13]{LOT11:HomPair} and the fact
  that for the antipodal pointed matched circle $\PMC$, $\PMC=\PMC_*$. So, only
  the case $k=2$, $i=0$ remains. This can be checked by computer,
  as follows. The proof of Proposition~\ref{prop:Alg-is-hom-smooth}
  gives a small model for the bar complex (first appearing in~\cite[Section
  5.4]{LOT11:HomPair}), which in
  turn gives a model for the Hochschild cochain complex of
  $\Alg(\PMC,0)$. Explicitly, this cochain complex is
  $\Alg(\PMC,0)\otimes \Alg(-\PMC,0)$, with differential given by 
  \[
  \bdy (x\otimes y)=\sum_{\text{chords }\xi}x a(\xi)\otimes a(\xi)
  y+a(\xi) x\otimes y a(\xi).
  \]
  There is an analogous model for
  $\HC_*(\Alg(\PMC,0)^!\DTP\Alg(\PMC,0)^!)$.  We are then interested in
  repeatedly applying $\bdy$ and $(1+\tau)^{-1}$ to the element
  $e_0 := (1|1|1|1)$, as in the proof of Theorem \ref{thm:HF-pi-formal-extreme}.
  A computer calculation then gives the following:
  \begin{itemize}
  \item $\bdy e_0 \in \HC_{-1}(\Alg(\PMC,0)^!\DTP\Alg(\PMC,0)^!)$ 
            is supported on 192 basis elements, and $\bdy e_0 = (1+\tau)(e_1)$
            for an element $e_1 \in \HC_{-1}$ supported on 96 basis elements.
  \item $\bdy(e_1) \in \HC_{-2}$ is supported on $1176$ basis elements,
           and $\bdy(e_1) = (1+\tau)(e_2')$ for an element $e_2' \in \HC_{-2}$
           supported on 588 basis elements.  (We eventually
           have to modify this lift of $\bdy(e_1)$, which is why we call it $e_2'$.)
  \item $\bdy(e_2') \in \HC_{-3}$ is supported on $2106$ elements, and
            $\bdy(e_2') = (1+\tau)(e_3')$ for an element $e_3' \in \HC_{-3}$
            supported on $1053$ basis elements.  However $\bdy(e_3')$ is not
            in the image of $(1+\tau)$.
  \item There is an element $x \in \HC_{-2}$ which is supported on 16
            ``square'' basis elements (elements of the form $(a|b|a|b)$), and
            $e_2 := e_2' + x$ has $(1+\tau)(e_2) = \bdy(e_1)$ and $\bdy(e_2) \in \HC_{-3}$
            is supported on 2250 elements.  Moreover $\bdy(e_2) = (1+\tau)(e_3)$
            for an element $e_3$ supported on $1125$ basis elements.
  \item $\bdy e_3 \in \HC_{-4}$ is supported on $3092$ basis elements.  Moreover
            $\bdy e_3 = (1+\tau)(e_4')$ for an element $e_4' \in \HC_{-4}$ supported on
            $1546$ basis elements.  This shows that the differential $d^4$ vanishes
            on $(1|1|1|1)$. 
  \item $\bdy e_4' \in \HC_{-5}$ is supported on $1944$ basis elements, and
            $\bdy e_4' = (1+\tau)(e_5')$ for $e_5' \in \HC_{-5}$ supported on $972$
            basis elements.  However, $\bdy e_5'$ is not in the image of $(1+\tau)$.
  \item There is an element $y \in \HC_{-4}$ supported on $24$ square basis elements,
            and $e_4 = e_4' + y$ has $(1+\tau)(e_4) = \bdy(e_3)$ and $\bdy(e_4) \in \HC_{-5}$
            is supported on $2048$ basis elements.  Moreover $\bdy(e_4) = (1+\tau)(e_5)$
            for an element $e_5$ supported on $1024$ basis elements.
  \item $\bdy e_5$ is supported on $788$ basis elements, and $\bdy e_5 = (1+\tau)(e_6)$
            for an element $e_6$ supported on $394$ basis elements.  This shows
            that $d^6$ vanishes on $(1|1|1|1)$.
 \end{itemize}
 The same computer code can be used to find $\HH_j(\Alg(\PMC,0)^!)$, in fact
 \[
 \HH_0 = \Field \quad  \HH_{-1} = \mathbb{F}_2^4 \quad \HH_{-2} = \mathbb{F}_2^{10} \quad \HH_{-3} = \Field
 \]
 and all other groups vanish.  By Proposition \ref{prop:TateE2}, it follows that $d^{2i}$ vanishes for $i > 3$ and the Theorem is proved.
  Computer code is available from 

\noindent
\verb|http://math.columbia.edu/~lipshitz/BordHochLoc.tar|.
\end{proof}

We conclude this section by observing that to obtain localization
results, it suffices to show that the relevant bimodules are neutral
(Definition~\ref{def:neutral}):
\begin{proposition}\label{prop:Borromean}
  For any pointed matched circle $\PMC$ and any integer $i$, the
  Hochschild homology $\HH_*(\Alg(\PMC,i))$ is supported in a single
  grading.
\end{proposition}
\begin{proof}
  Suppose that $\PMC$ represents a surface of genus $k$.  By
  Theorem~\ref{thm:bordered-hochschlid}, $\HH_*(\Alg(\PMC,i))$ is the
  knot Floer homology of the $k\th$ Borromean knot (in
  $\#^{2k}(S^2\times S^1)$) in the $i\th$
  Alexander grading. So, the result follows from the computation of
  $\HFKa(B_k)$~\cite[Section 9]{OS04:Knots}.
\end{proof}
\begin{corollary}\label{cor:neutral-good-bord}
  Every neutral $(\Alg(\PMC),\Alg(\PMC))$-bimodule is $\pi$-formal.
\end{corollary}
\begin{proof}
  This is immediate from Proposition~\ref{prop:Borromean} and
  Corollary~\ref{cor:neutral-formal}.
\end{proof}

\subsection{Branched double covers of links}\label{sec:br-d-cov}
As an expository point, we explain how
Conjecture~\ref{conj:HF-pi-formal} implies~\cite[Theorem 1.1]{Hendricks12:rank-inequal}:
\begin{theorem}\label{thm:br-d-cov-sseq}
  Let $K\subset S^3$ be a nullhomologous knot and $\pi\co \Sigma(K)\to
  S^3$ the double cover of $S^3$ branched along $K$.  Suppose that
  Conjecture~\ref{conj:HF-pi-formal} holds for some arc diagram $\PMC$
  representing a Seifert surface for $K$. Then there is a spectral
  sequence with $E^1$-page given by $\HFKa(\Sigma(K),\pi^{-1}(K))$
  converging to $\HFKa(S^3,K)$.
\end{theorem}
\begin{proof}
  This follows easily from Theorem~\ref{thm:hoch-local} and
  Theorem~\ref{thm:bordered-hochschlid}. Let $F\subset S^3$ be a
  Seifert surface for $K$ and let $Y=S^3\setminus\nbd(F)$. Choose a
  homeomorphism $\phi\co \PunctF(\PMC)\to F$. Let $C\subset \bdy Y$ be a push-off of $\bdy
  F$ and let $Y_1$ be the result of attaching a $3$-dimensional
  $2$-handle (thickened disk) to $Y$ along $C$. The manifold $Y_1$ has
  two boundary components $\bdy_LY_1$ and $\bdy_RY_1$, and the co-core of the new
  $2$-handle gives a framed arc $\arcz$ in $Y_1$ connecting
  $\bdy_LY_1$ and $\bdy_RY_1$. The map $\phi$
  induces homeomorphisms $\phi_L\co {-F(\PMC)}\to \bdy_LY_1$ and
  $\phi_R\co F(\PMC)\to\bdy_R Y_1$. The data
  $(Y_1,\phi_L,\phi_R,\arcz)$ is an arced
  cobordism from $F(\PMC)$ to itself; abusing
  notation, we will denote this arced cobordism by
  $Y_1$.

  It is immediate from the definition that the generalized open book
  $(Y_1^\circ, L)$ induced by $Y_1$ is exactly $(S^3,K)$. Thus, by Theorem~\ref{thm:bordered-hochschlid},
  \[
  \HFKa(S^3,K)\cong\HH_*(\CFDAa(Y_1)).
  \]
  Let $Y_2=Y_1\sos{\bdy_R}{\cup}{\bdy_L}Y_1$. Then the generalized
  open book $(Y_2^\circ,L)$ associated to $Y_2$ is exactly
  $(\Sigma(K),\pi^{-1}(K))$. Thus,
  \[
  \HFKa(\Sigma(K),\pi^{-1}(K))=\HH_*(\CFDAa(Y_1)\DTP_{\Alg(\PMC)}\CFDAa(Y_1)).
  \]
  So, in light of Proposition~\ref{prop:Alg-is-hom-smooth}, the result
  follows from Theorem~\ref{thm:hoch-local}.
\end{proof}

\begin{thmcor}\label{cor:knot-s-seq-g2-S3}
  If $K\subset S^3$ has a Seifert surface of genus $\leq 2$ then there
  is a spectral sequence
  \[
  \HFKa(\Sigma(K),\pi^{-1}(K))\Rightarrow \HFKa(S^3,K).
  \]
\end{thmcor}
\begin{proof}
  This is immediate from Theorems~\ref{thm:HF-pi-formal-small}
  and~\ref{thm:br-d-cov-sseq}.
\end{proof}

It is not hard to show that Theorem~\ref{thm:br-d-cov-sseq} respects
the $\SpinC$-structure and Alexander grading as
in~\cite{Hendricks12:rank-inequal}. Rather than spelling this out here,
we turn to a generalization of Theorem~\ref{thm:br-d-cov-sseq}, and
spell out the analogous issues in the generalization. To state the
generalization, we digress briefly to discuss branched double covers
of nullhomologous links in other $3$-manifolds. 

Let $Y$ be a
$3$-manifold and $L\subset Y$ a nullhomologous link. Fix a Seifert
surface $F$ for $L$. Then $F$ is Poincar\'e-Lefschetz dual to an
element of $H^1(Y\setminus L)$, which we can view as a map $\ell_F \co
H_1(Y\setminus L)\to \ZZ$.
The composition 
\[
\pi_1(Y\setminus L)\to H_1(Y\setminus L)\to \ZZ\to \ZZ/2
\]
defines a $2$-fold cover of $p\co \widetilde{Y\setminus L}\to
Y\setminus L$. Write the components of $L$ as $L_1,\dots, L_n$, and
let $\mu_i$ be a meridian of $L_i$. Then each $L_i$ corresponds to a
torus boundary component $T_i$ of $\widetilde{Y\setminus \nbd(L)}$. Fill in
$T_i$ with a solid torus in such a way that $p^{-1}(\mu_i)$ bounds a
disk. The result is a closed $3$-manifold $\Sigma(L)$, the double
cover of $Y$ branched along $L$, and a map $\pi\co \Sigma(L)\to
Y$. While $\pi$ does depend on $F$, through its relative homology class,
we will suppress $F$ from the notation.

We digress briefly to discuss $\SpinC$-structures. Consider
$Y\setminus\nbd(L)$. There is a unique up to isotopy non-vanishing
vector field $v_0$ in $T(\bdy\nbd(L))$ so that $v_0$ is everywhere
transverse to a meridian for (the relevant component of) $L$. A
\emph{relative $\SpinC$-structure} for $(Y,L)$ is a homology class of
vector fields $v$ on $Y\setminus \nbd(L)$ so that
$v|_{\bdy\nbd(L)}=v_0$; compare~\cite[Section 3.2]{OS05:HFL}. Let
$\SpinC(Y,L)$ denote the set of relative $\SpinC$-structures on
$(Y,L)$.  (It is worth noting that the vector field $v_0$ used here
and in~\cite{OS05:HFL} is different from, but isotopic in $TY|_{\bdy\nbd(L)}$ to, the
analogous vector field $v_0$ that arises in sutured Floer
homology~\cite[Section 4]{Juhasz06:Sutured}.)

Since $v_0$ pulls back to $v_0$ under the branched double cover map
$\pi\co \Sigma(L)\to Y$, there is a map $\pi^*\co \SpinC(Y,L)\to
\SpinC(\Sigma(L),\pi^{-1}(L))$. Since $c_1$ is natural, The map
$\pi^*$ sends torsion $\SpinC$ structures---i.e., $\SpinC$ structures
whose first Chern classes are torsion---to torsion $\SpinC$
structures. On a related point, the involution $\tau\co
(\Sigma(L),\pi^{-1}(L))\to(\Sigma(L),\pi^{-1}(L))$ of the branched
double cover induces an involution $\tau_*\co
\SpinC(\Sigma(L),\pi^{-1}(L))\to \SpinC(\Sigma(L),\pi^{-1}(L))$. The
image of $\pi^*\co \SpinC(Y,L)\to \SpinC(\Sigma(L),\pi^{-1}(L))$ is
contained in the fixed set of $\tau_*$.

Recall that $\HFLa(Y,L)$ decomposes as a direct sum 
\[
\HFLa(Y,L)=\bigoplus_{\spinc\in\SpinC(Y, L)}\HFLa(Y,L;\spinc),
\]
Each $\HFLa(Y,L;\spinc)$ has a relative grading by
$\ZZ/\divis(c_1(\spinc))$, where $c_1(\spinc)$ denotes the first Chern
class of the $2$-plane field associated to $\spinc$ and $\divis$
denotes the divisibility of the cohomology class $c_1(\spinc)$, i.e.,
$\divis(a)=\max\{n\in\ZZ\mid \exists b, a=n\cdot b\}$. In particular,
$\HFLa(Y,L;\spinc)$ is relatively $\ZZ$-graded exactly when
$c_1(\spinc)$ is torsion. The relevance of this condition is that Theorem~\ref{thm:hoch-local}
needs the Hochschild chain complex to be $\ZZ$-graded.

Given a Seifert surface $F$ for $L$ there is a corresponding surface
$F^\circ$ inside the $0$-surgery $Y_0(L)$.  Similarly, given a
relative $\SpinC$-structure $\spinc\in \SpinC(Y,L)$ there is a
corresponding $\SpinC$-structure $\spinc^\circ\in \SpinC(Y_0(L))$.
Given an absolute $\SpinC$-structure $\mathfrak{t}\in\SpinC(Y\setminus
L)$, let
\[
\HFLa(Y,L;\spinct,i)=\bigoplus_{\substack{\spinc\in\SpinC(Y,L)\\\spinc|_{Y\setminus
      L}=\spinct\\\langle c_1(\spinc^\circ),
    F^\circ\rangle=2i}}\HFLa(Y,L;\spinc).
\]
Note that, even though it does not appear in the notation,
$\HFLa(Y,L;\spinct,i)$ depends on $F$.

We are now ready for the promised generalization of Theorem~\ref{thm:br-d-cov-sseq}:
\begin{theorem}\label{thm:br-d-cov-sseq-gen}
  Let $Y^3$ be a closed $3$-manifold, $L\subset Y$ a nullhomologous
  link and $\spinc$ a torsion $\SpinC$-structure on $Y\setminus
  L$. Let $F$ be a Seifert surface for $L$. Suppose that
  Conjecture~\ref{conj:HF-pi-formal} holds for a pointed matched
  circle $\PMC$ representing $F$ and an integer $i$.  Then there is a
  spectral sequence with $E^1$-page given by
  $\HFLa(\Sigma(L),\pi^{-1}(L);\pi^*\spinct,i)$ converging to
  $\HFLa(Y,L;\spinct, i)$. The $d^j$ differential in this spectral
  sequence increases the (relative) Maslov grading by $j-1$.
\end{theorem}
\begin{proof}[Proof of Theorem~\ref{thm:br-d-cov-sseq-gen}]
  The proof is essentially the same as the proof of
  Theorem~\ref{thm:br-d-cov-sseq}, after replacing $S^3$ by $Y$, so we
  will be brief.  Let $Y_1$ denote the result of cutting $Y$ along a Seifert
  surface $F$ for $L$. The boundary of $Y_1$ is divided naturally into
  three parts: a copy of $F$, a copy of $-F$, and $\amalg_{i=1}^n
  [0,1]\times S^1$. Make $F$ into a sutured surface by dividing each
  boundary component into two connected arcs $S_\pm$, and choose an
  arc diagram $\PMC$ and diffeomorphism $\phi\co \PunctF(\PMC)\to
  F$. Identifying $\{0\}\times S^1$ (respectively $\{1\}\times S^1$)
  with $\bdy F$, let $R_\pm = S_\pm\times[0,1]\subset
  S^1\times[0,1]$. This makes $Y_1$ into a sutured cobordism from
  $\PunctF(\PMC)$ to itself. 
  By Theorem~\ref{thm:sutured-self-pairing}, and the interpretation of
  sutured Floer homology of a link complement as link Floer
  homology~\cite[Proposition 9.2]{Juhasz06:Sutured},
  \begin{align*}
    \HFLa(Y,L)&\cong \HH_*(\BSDAa(Y_1))\\
    \HFLa(\Sigma(L),\pi^{-1}(L))&\cong \HH_*(\BSDAa(Y_1)\DTP\BSDAa(Y_1)).
  \end{align*}

  The behavior of $d^i$ on the relative Maslov grading is obvious from
  the construction of the spectral sequence
  (cf.~Remark~\ref{remark:vh-pictures}).
  So, if we ignore the decomposition into $\SpinC$-structures (and the
  corresponding issues with the $\ZZ$-grading), the result follows
  from Theorem~\ref{thm:hoch-local} (using
  Proposition~\ref{prop:Alg-is-hom-smooth}).

  There are two options for treating the $\SpinC$-structures: either
  we can study carefully the $G$-set valued gradings on $\BSDAa$ and
  in the pairing theorem or we can look back at the proof of
  Theorem~\ref{thm:hoch-local}. We will explain the latter option.

  Let $M$ denote $\BSDAa(Y_1)$ and consider the bicomplex
  $\HC_{*,*}^\Tate(M\DTP M)$. By the self-pairing theorem
  (Theorem~\ref{thm:bordered-hochschlid}), each column in
  $\HC_{*,*}^\Tate(M\DTP M)$ is homotopy equivalent to
  $\CFLa(\Sigma(L),\pi^{-1}(L))$. The vertical differentials in the bicomplex respect
  the decomposition of $\CFLa(\Sigma(L),\pi^{-1}(L))$ into relative
  $\SpinC$-structures. The horizontal differentials do not respect the
  decomposition, but do respect the decomposition into $\tau_*$-orbits
  of relative $\SpinC$-structures,
  \[
  \CFLa(\Sigma(L),\pi^{-1}(L))=\bigoplus_{\spinc\in\SpinC(\Sigma(L),\pi^{-1}(L))/\tau_*}\CFLa(\Sigma(L),\pi^{-1}(L);\spinc)\oplus
  \CFLa(\Sigma(L),\pi^{-1}(L);\tau_*\spinc).
  \]
  It follows that the entire spectral sequence decomposes into
  $\tau_*$-orbits of relative $\SpinC$-structures.
  It remains to verify that the isomorphism $\HFLa(Y,L)\cong
  E^3_{p,*}$ respects relative $\SpinC$-structures, in the sense that
  for each relative $\SpinC$-structure $\spinc$ the isomorphism
  identifies $\HFLa(Y,L;\spinc)$ with $E^3_{p,*}(\pi^*\spinc)$. This,
  in turn, follows from the fact that given a generator $\x$ for
  $\HFLa(Y,L)\cong\HH_*(M)$ representing the $\SpinC$-structure
  $\spinc$, $\x\otimes\x\in \HH_*(M\DTP M)\cong
  \HFLa(\Sigma(L),\pi^{-1}(L))$ represents the $\SpinC$-structure
  $\pi^*\spinc$, which is immediate from how a $\SpinC$-structure
  is associated to a generator (see~\cite[Section 3.6]{OS05:HFL}).
\end{proof}

\begin{proof}[Proof of Theorem~\ref{thm:br-g2}]
  This is immediate from Theorems~\ref{thm:HF-pi-formal-small}
  and~\ref{thm:br-d-cov-sseq-gen}.
\end{proof}

\begin{proof}[Proof of Theorem~\ref{thm:br-extreme}]
  This is immediate from Theorems~\ref{thm:HF-pi-formal-extreme}
  and~\ref{thm:br-d-cov-sseq-gen}.
\end{proof}

\subsection{The tube-cutting piece}\label{sec:TC}
To use Theorem~\ref{thm:hoch-local} to obtain results about the
Heegaard Floer homology of closed $3$-manifolds we need a Hochschild
homology interpretation of $\HFa$ (rather than $\HFKa$). This is
obtained by using a bimodule associated to a particular bordered
Heegaard diagram, which we call the \emph{tube-cutting piece}.

\begin{definition}
  Let $\PMC$ be a pointed matched circle or, more generally, arc diagram. The \emph{tube-cutting
    piece} for $\PMC$, denoted $\TC(\PMC)$, is the bordered-sutured
  Heegaard diagram defined as follows. Let $\Id(\PMC)$ denote the
  standard Heegaard diagram for the identity map of $F(\PMC)$; see
  \cite[Definition 5.35]{LOT2} or Figure~\ref{fig:tube-cutting}. Write
  $\Id(\PMC)=(\Sigma,
  \{\alpha_1^a,\dots,\alpha_{2k}^a\},\{\beta_1,\dots,\beta_k\},\arcz)$. The
  surface $\Sigma$ has two boundary components $\bdy_L\Sigma$ and
  $\bdy_R\Sigma$, and $z$ is an arc connecting $\bdy_L\Sigma$ and
  $\bdy_R\Sigma$. Let $\alpha_1^c$ (respectively $\beta_{k+1}$) be an
  embedded circle in $\Sigma$ disjoint from the $\alpha_i^a$
  (respectively $\beta_i$) and homologous to $\bdy_R\Sigma$. Let
  $z_1=\arcz\cap\bdy_L\Sigma$ and $z_2=\arcz\cap\bdy_R\Sigma$. Then
  $\TC(\PMC)=(\Sigma,\{\alpha_1^a,\dots,\alpha_{2k}^a,\alpha_1^c\},\{\beta_1,\dots,\beta_k,\beta_{k+1}\},\{z_1,z_2\})$.
\end{definition}

We turn next to the topological interpretation of $\TC(\PMC)$.
Recall that $\PMC$
specifies a surface $\PunctF(\PMC)$ with a single boundary
component. Bordered-sutured Floer theory interprets the diagram
$\Id(\PMC)$ as representing $[0,1]\times \PunctF(\PMC)$. The boundary
of $[0,1]\times \PunctF(\PMC)$ is divided into three pieces:
$\{0,1\}\times\PunctF(\PMC)$, which are viewed as bordered boundary
(i.e., boundary that one can glue along) and $[0,1]\times
(\bdy\PunctF)$, which is sutured boundary, with two longitudinal
sutures running along it. The diagram $\TC(\PMC)$ represents the
result of attaching a $2$-handle to $[0,1]\times (\bdy\PunctF)$ along
$\{1/2\}\times (\bdy\PunctF)$, and placing sutures on the result in
the obvious way.

\begin{figure}
  \centering
  \includegraphics{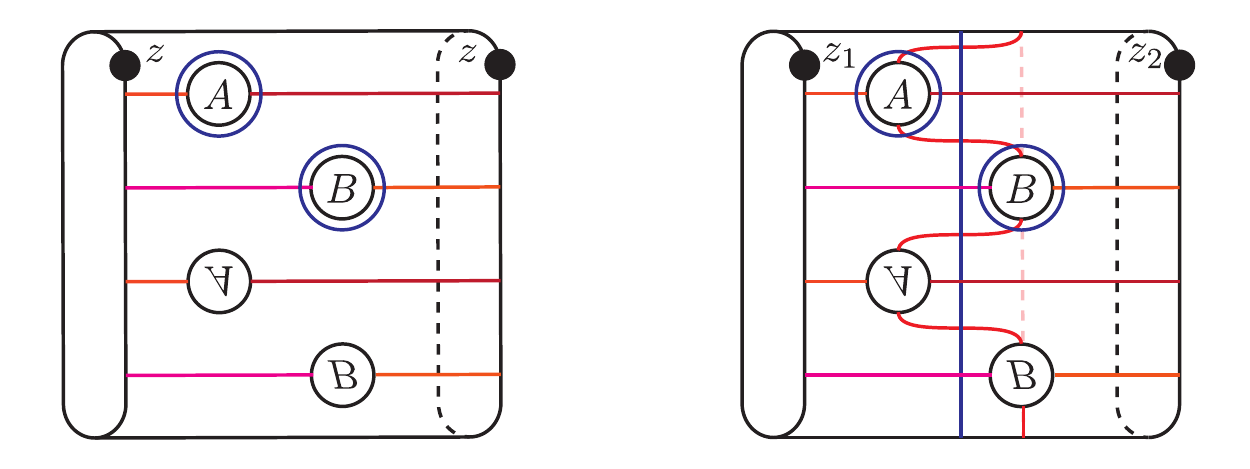}
  \caption{\textbf{The tube-cutting piece.} The diagram illustrates
    the genus $1$ case. Left: the standard bordered Heegaard diagram
    for the identity map of the torus. Right: the bordered Heegaard
    diagram $\TC$.}
  \label{fig:tube-cutting}
\end{figure}

\begin{theorem}\label{thm:hoch-TC}
  Let $\HD$ be a bordered Heegaard diagram for an arced cobordism $Y$
  from $F(\PMC)$ to itself.
  Let $T_Y$ denote the closed $3$-manifold
  obtained by gluing the two boundary components of $Y$ together in
  the obvious way, i.e., 
  \[
  T_Y= Y/\bigl(F(\PMC)\ni x\sim x\in -F(\PMC)\bigr).
  \]
  Then 
  \[
  \HH_*(\BSDAa(\HD\sos{F(\PMC)}{\cup}{-F(\PMC)}\TC(\PMC)))\cong \HFa(T_Y).
  \]
\end{theorem}
\begin{proof}
  Let $\HD'$ be the sutured Heegaard diagram obtained by gluing $\HD$ to the tube-cutting piece $\TC(\PMC)$ along both boundary components.  
  From the self-pairing theorem for bordered-sutured Floer
  homology, Theorem~\ref{thm:sutured-self-pairing},
  \[
  H_*(\BSAa(\HD)\DTP_{\Alg(\PMC)\otimes\Alg(-\PMC)}\BSDa(\TC(\PMC))=\SFH(\HD').  
  \]
  From the topological interpretation of $\TC(\PMC)$ and gluing
  properties of bordered-sutured diagrams \cite[Proposition
  4.15]{Zarev09:BorSut}, $\HD'$ is a sutured Heegaard diagram for
  $T_Y\setminus B^3$ with a single suture on the $S^2$ boundary
  component. Thus,
  \[
  \SFH(\HD')=\HFa(T_Y)
  \]
  (see~\cite[Proposition 9.1]{Juhasz06:Sutured}).
\end{proof}

We will also use a variant of the tube-cutting piece in order to prove
that certain bimodules are neutral
(Definition~\ref{def:neutral}). Consider the Heegaard diagram
$\TC(\PMC)$. Draw an arc $\gamma$ from $z_1$ to $z_2$ in
$
\Sigma\setminus(\alpha_1^a\cup\dots\cup\alpha_{2k}^a\cup\beta_1\cup\dots\cup\beta_k).
$
Choose a point $z_3$ on $\gamma$, dividing $\gamma$ into two subarcs
$\gamma_{13}$ from $z_1$ to $z_3$ and $\gamma_{32}$ from $z_3$ to
$z_2$. Choose $z_3$ so that $\gamma_{13}$ intersects $\beta_{2k+1}$
once and is disjoint from $\alpha_1^c$, while $\gamma_{32}$ intersects
$\alpha_1^c$ once and is disjoint from $\beta_{2k+1}$. (See
Figure~\ref{fig:TCc}. It may be necessary to perturb $\alpha_1^c$,
$\beta_{2k+1}$ and $\gamma$ in order to be able to choose $z_3$ this
way.) Let
\[
\TCc(\PMC)=(\Sigma,\{\alpha_1^a,\dots,\alpha_{2k}^a,\alpha_1^c\},\{\beta_1,\dots,\beta_k,\beta_{k+1}\},\{z_1,z_2,z_3\}).
\]
We can again view $\TCc(\PMC)$ as a bordered-sutured Heegaard diagram,
now representing $[0,1]\times \PunctF(\PMC)$ with sutures on
$[0,1]\times(\bdy\PunctF(\PMC))$ as shown in Figure~\ref{fig:TCc}.

\begin{figure}
  \includegraphics{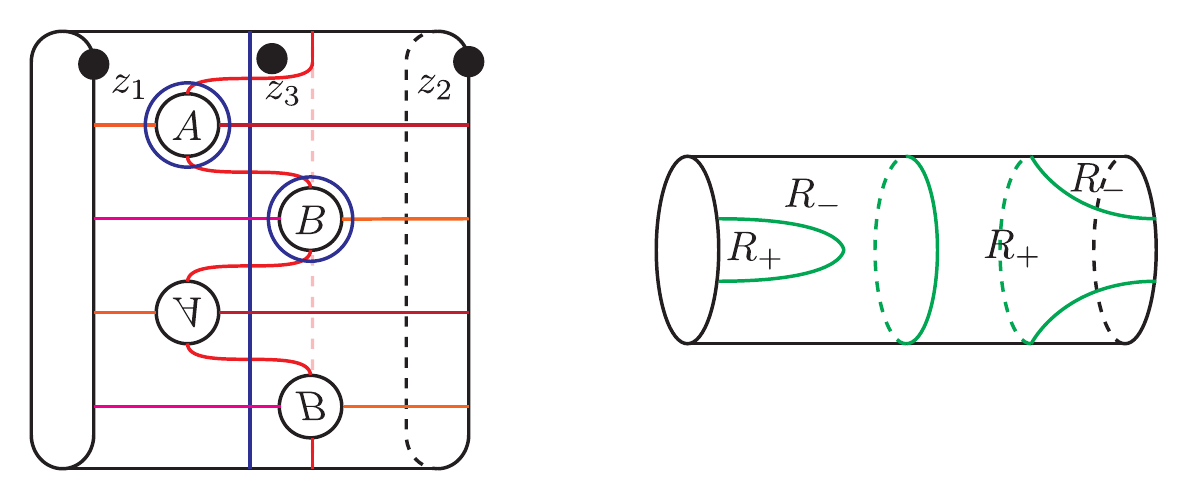}
  \caption{\textbf{The diagram $\TCc(\PMC)$ and the corresponding
      bordered-sutured manifold.} Left: the genus $1$ case of the
    diagram $\TCc(\PMC)$. Right: the corresponding sutures on
    $[0,1]\times S^1\subset \bdy([0,1]\times\PunctF(\PMC))$.}
  \label{fig:TCc}
\end{figure}

We are interested in $\TCc(\PMC)$ because of two key properties. First:
\begin{proposition}\label{prop:DTP-vanish}
  The Heegaard diagram $\TC(\PMC)\cup_{\PMC}\TCc(\PMC)$ has trivial
  bordered-sutured invariants. In particular,
  $\BSDAa(\TC(\PMC))\DTP_{\Alg(\PMC)}\BSDAa(\TCc(\PMC))$ is acyclic.
\end{proposition}
\begin{proof}
  One can perform a sequence of handleslides of the circle
  $\beta_{2k+1}$ in $\TCc(\PMC)$ over other $\beta$-circles in
  $\TC(\PMC)\cup_{\PMC}\TCc(\PMC)$, followed by an isotopy, so that
  the resulting circle $\beta_{2k+1}'$ is a small circle around $z_3$
  disjoint from the $\alpha$-curves. Moreover, because of the
  placement of the basepoints, this diagram is still admissible. So,
  $\TC(\PMC)\cup_{\PMC}\TCc(\PMC)$ is equivalent to an admissible
  diagram in which there are no generators for the bordered-sutured
  invariants; this implies that the bordered-sutured invariants of
  $\TC(\PMC)\cup_{\PMC}\TCc(\PMC)$ are trivial.
\end{proof}

Let $\tau_\bdy$ denote a positive Dehn twist of $\PunctF(\PMC)$ along
a curve parallel to the boundary (``the boundary Dehn twist'') and
$Y_{\tau_\bdy}$ the mapping cylinder of $\tau_\bdy$. Let
$\tau_\bdy^{-1}$ and $Y_{\tau_\bdy^{-1}}$ denote the negative boundary
Dehn twist and its mapping cylinder.

\begin{figure}
  \centering
  \includegraphics{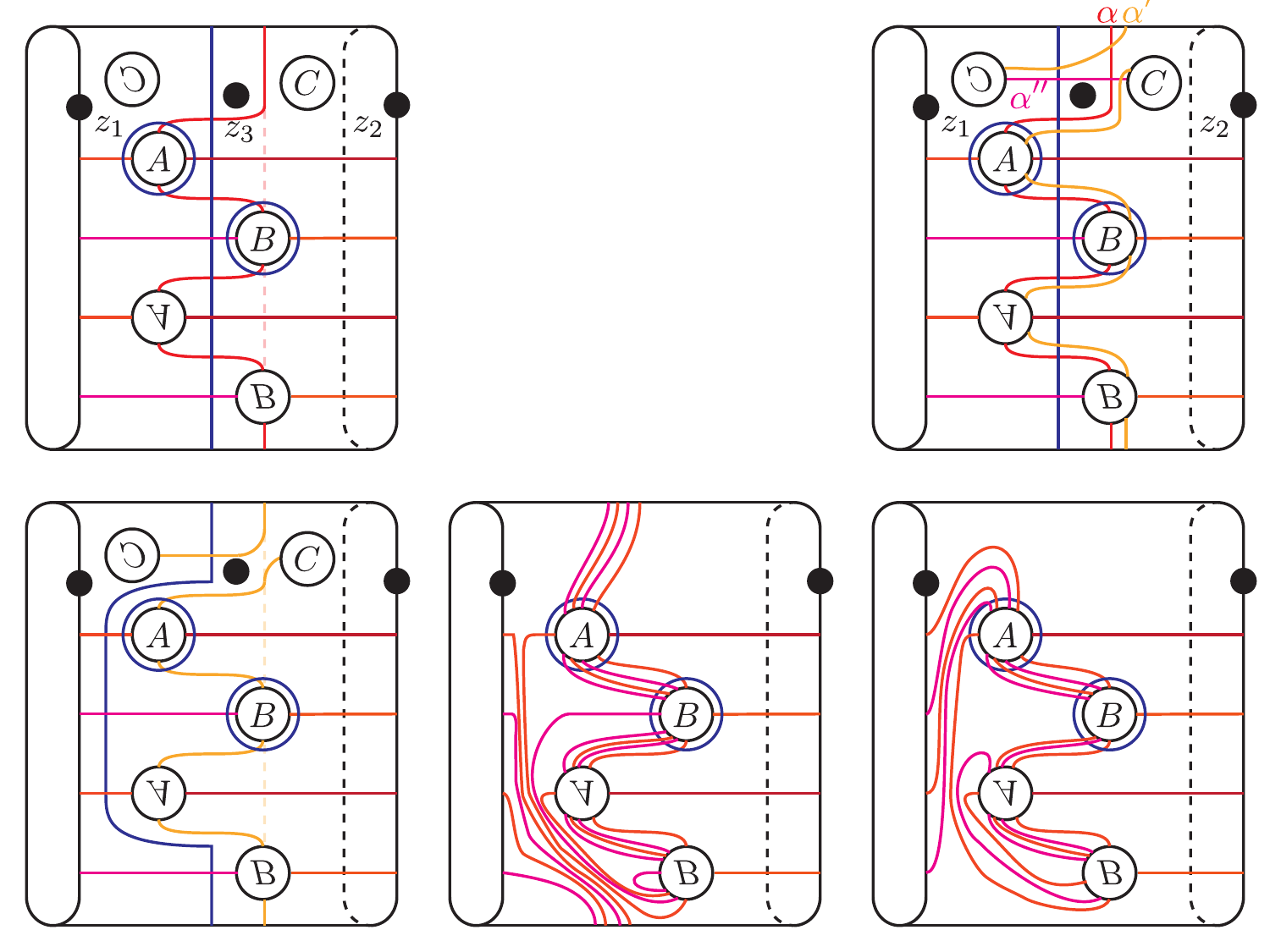}
  \caption{\textbf{Exact triangle for $\TCc(\PMC)$}. Top left: the
    bordered-sutured Heegaard diagram $\TCc(\PMC)$ with an extra
    handle. Top right: the bordered-sutured triple diagram inducing the exact
    triangle in Theorem~\ref{thm:bypass-tri}. Bottom: Identifying the
    third term in the exact triangle with the positive boundary Dehn
    twist. The bottom left and bottom center differ by a sequence of
    handleslides and a destabilization; the bottom right is obtained
    by applying a surface diffeomorphism to the bottom center.}
  \label{fig:TC-exact-tri}
\end{figure}

The second key property of $\TCc(\PMC)$ is:

\begin{theorem}\label{thm:bypass-tri}
  There are exact triangles
  \[
  \xymatrix{
    \BSDAa(\TCc(\PMC)) \ar[rr] & & \BSDAa(Y_{\tau_\bdy})\ar[dl]\\
    & \BSDAa(\Id) \ar[ul]& 
  }
  \]
  and
  \[
  \xymatrix{
    \BSDAa(\TCc(\PMC))\ar[rr] & & \BSDAa(\Id)\ar[dl]\\
    & \BSDAa(Y_{\tau_\bdy^{-1}}). \ar[ul]& 
  }
  \]
\end{theorem}
\begin{proof}
  We construct a bordered-sutured quadruple Heegaard diagram
  $(\Sigma,\alphas,\alphas',\alphas'',\betas,\{z_1,z_2,z_3\})$ with the following properties:
  \begin{enumerate}
  \item $\BSDAa(\Sigma,\alphas,\betas,\{z_1,z_2,z_3\})\cong
    \BSDAa(\TCc(\PMC))$. (In fact, the bordered-sutured 3-manifolds
    specified by $(\Sigma,\alphas,\betas,\{z_1,z_2,z_3\})$ and
    $\TCc(\PMC)$ differ by a product decomposition.)
  \item\label{item:is-diag-tou}
    $(\Sigma,\alphas',\betas,\{z_1,z_2,z_3\})$ is a bordered-sutured
    Heegaard diagram for $\tau_\bdy$.
  \item\label{item:is-diag-id}
    $\BSDAa(\Sigma,\alphas'',\betas,\{z_1,z_2,z_3\})\simeq
    \Alg(\PMC)$. (In fact, $(\Sigma,\alphas'',\betas,\{z_1,z_2,z_3\})$
    is a bordered Heegaard diagram for the mapping cylinder of the
    identity map.)
  \item Each of $\alphas$, $\alphas'$ and $\alphas''$ consists of $2k$
    arcs and $1$ circle.
  \item The arcs in $\alphas$, $\alphas'$ and $\alphas''$ are
    the same.
  \item\label{item:local-pic} Let $\alpha$, $\alpha'$ and $\alpha''$
    denote the circles in $\alphas$, $\alphas'$ and $\alphas''$,
    respectively. Then $\alpha$, $\alpha'$ and $\alpha''$ all lie in a
    punctured torus $T$ in $\Sigma$ disjoint from the $\alpha$-arcs,
    and with respect to an appropriate orientation-preserving
    identification of $T$ with $\RR^2/\ZZ^2$, $\alpha$ corresponds to
    the line $x=0$, $\alpha'$ corresponds to the line $y=x$ and
    $\alpha''$ corresponds to the line $x=0$.  (That is, $\alpha$,
    $\alpha'$ and $\alpha''$ have slopes $\infty$, $1$ and $0$,
    respectively.)
  \end{enumerate}
  The first exact triangle then follows from the pairing theorem in
  bordered-sutured Floer homology and the exact triangle of type $D$
  invariants in~\cite[Section 11.2]{LOT1}. (The strange cyclic
  ordering $\infty$--$1$--$0$ comes from the fact that we are varying
  the $\alpha$-circles, not the $\beta$-circles.)

  The quadruple diagram is illustrated in
  Figure~\ref{fig:TC-exact-tri}.  To construct it, start with the
  bordered Heegaard diagram
  $\TCc(\PMC)=(\Sigma',\alphas,\betas,\{z_1,z_2,z_3\})$. Add a new
  handle with one foot near $z_1$ and one foot near $z_3$; call the
  resulting surface $\Sigma$. Since both feet of the new handle are in
  regions containing basepoints,
  $\BSDAa(\Sigma,\alphas,\betas,\{z_1,z_2,z_3\})\cong
  \BSDAa(\TCc(\PMC))$ (and, in fact, the corresponding
  bordered-sutured 3-manifolds differ by a disk decomposition).

  Let $\alpha=\alpha_1^c$ be the unique circle in $\alphas$. Let
  $\alpha''$ be a circle which runs along the new handle in $\Sigma$
  once, intersects $\alpha$ and $\beta_{2k+1}$ once each, and is
  disjoint from the other $\alpha$- and $\beta$-curves. Obtain
  $\alpha'$ from $\alpha''\cup \alpha$ by smoothing the unique
  crossing. There are two ways to perform this smoothing; one of the
  two gives curves satisfying property~(\ref{item:local-pic}).

  It remains to verify properties~(\ref{item:is-diag-tou})
  and~(\ref{item:is-diag-id}). Property~(\ref{item:is-diag-id}) is
  easy: since the only $\beta$-circle that $\alpha''$ intersects is
  $\beta_{2k+1}$, any generator for
  $\BSDAa(\Sigma,\alphas',\betas,\{z_1,z_2,z_3\})$ must contain this
  point. This gives an identification of generators for
  $\BSDAa(\Sigma,\alphas',\betas,\{z_1,z_2,z_3\})$ and $\BSDAa$ of the
  standard Heegaard diagram for the identity cobordism. Moreover, the
  placement of the basepoints means that exactly the same curves are
  counted in the $\Ainf$-structure on the two bimodules. (Alternately,
  one can destabilize $\alpha''$ and $\beta_{2k+1}$ to obtain the
  standard Heegaard diagram for the identity cobordism.)

  For Property~(\ref{item:is-diag-tou}), we manipulate the Heegaard
  diagram. Specifically, after performing a sequence of handleslides
  (two for each $\alpha$-arc on the left-hand side of the diagram,
  say) one can destabilize $\alpha'$ and $\beta_{2k+1}$ to obtain a
  Heegaard diagram for the boundary Dehn twist; see
  Figure~\ref{fig:TC-exact-tri} for the genus $1$ case.  (To be
  convinced of the sign of the Dehn twist, compare with~\cite[Figure
  12]{LOT11:HomPair} and count the number of intersection points on
  each $\alpha$-arc.)

  To obtain the second exact triangle, tensor the first with
  $\BSDAa(Y_{\tau_\bdy^{-1}})$, and note that
  $\TCc(\PMC)\cup_{\PunctF(\PMC)}Y_{\tau_\bdy^{-1}}\cong \TCc(\PMC)$.
\end{proof}

\begin{corollary}\label{cor:induce-iso}
  Let $f\co \BSDAa(\tau_\bdy)\to \BSDAa(\Id)$ be the map from
  Theorem~\ref{thm:bypass-tri}. Then the map 
  \[
  f\otimes\Id\co \BSDAa(Y_{\tau_\bdy})\DTP \BSDAa(\TC(\PMC))\to
  \BSDAa(\Id)\DTP \BSDAa(\TC(\PMC))
  \]
  is a quasi-isomorphism. In particular, for any
  $(\Alg(\PMC),\Alg(\PMC))$-bimodule $M$, $f$ induces an isomorphism
  \[
  f_*\co 
  \HH_*\bigl(\BSDAa(Y_{\tau_\bdy})\DTP \BSDAa(\TC(\PMC))\DTP M\bigr)\stackrel{\cong}{\longrightarrow}\HH_*\bigl(\BSDAa(\TC(\PMC))\DTP
  M\bigr)
  \]
\end{corollary}
\begin{proof}
  Tensor the exact triangle in Theorem~\ref{thm:bypass-tri} with
  $\BSDAa(\TC(\PMC))$ and apply Proposition~\ref{prop:DTP-vanish} to
  see that every third term vanishes.
\end{proof}

Tensoring with the bimodule $\BSDAa(Y_{\tau_\bdy^{-1}})$ is the Serre functor for the derived
category of $\Alg(\PMC)$-bimodules~\cite[Theorem 10]{LOT11:HomPair}. In particular:
\begin{citethm}\label{thm:bord-coho-is-ho}\cite[Corollary 11]{LOT11:HomPair}
  For any $(\Alg(\PMC),\Alg(\PMC))$ bimodule $M$ we have
  \[
  \HH^*(M)\cong \HH_*(\BSDAa(Y_{\tau_\bdy})\DTP M).
  \]
\end{citethm}

\begin{corollary}\label{cor:is-neutral}
  Let $Y$ be a bordered-sutured $3$-manifold with bordered boundary
  $(-\PunctF(\PMC))\amalg\PunctF(\PMC)$. Let $Y'$ be the result of
  gluing $\TC(\PMC)$ to $Y$ along one boundary component. Then
  $\BSDAa(Y')$ is a neutral bimodule.
\end{corollary}
\begin{proof}
  By Corollary~\ref{cor:induce-iso}, the map 
  \[
  f_*\co \HH_*(\BSDAa(Y_{\tau_\bdy})\DTP\BSDAa(Y'))\to \HH_*(\BSDAa(Y'))
  \]
  is an isomorphism. By Theorem~\ref{thm:bord-coho-is-ho}, 
  \[
  \HH_*(\BSDAa(Y_{\tau_\bdy})\DTP \BSDAa(Y'))\cong \HH^*(\BSDAa(Y')).
  \]
  It remains to see that the isomorphism $f_*$ is induced by an
  element of $\HH_*(\Alg(\PMC))$, i.e., a map $\Alg(\PMC)^!\to
  \Alg(\PMC)$. But by~\cite[Proposition 5.13]{LOT11:HomPair},
  $\BSDAa(Y_{\tau_\bdy})\simeq \Alg(\PMC)^!$, so $f$ is indeed a map
  $\Alg(\PMC)^!\to \Alg(\PMC)$.
\end{proof}

\begin{remark}
  Theorem~\ref{thm:bypass-tri} can be seen as a special case of
  Honda's bypass exact triangle (in the bordered-sutured setting).
  Proposition~\ref{prop:DTP-vanish} can be deduced from the fact
  that a particular contact structure near $[0,1]\times
  \bdy\PunctF(\PMC)$ is overtwisted.
\end{remark}
\subsection{Double covers of \texorpdfstring{$3$}{3}-manifolds}\label{subsec:double-cover}
We turn next to a rank inequality for a class of (unbranched) double
covers. To spell out that class, recall that a double cover $\pi\co
\tilde{Y}\to Y$ corresponds to a homomorphism $p\co \pi_1(Y)\to
\ZZ/2$, which we can regard as an element $p\in H^1(Y,\ZZ/2)$. There
is a canonical change-of-coefficient homomorphism $c\co H^1(Y,\ZZ)\to
H^1(Y,\ZZ/2)$.
\begin{definition}\label{def:from-Z-cover}
  If $p$ is in the image of $c$ then we will say that $\pi$ is
  \emph{induced by a $\ZZ$-cover}.
\end{definition}

\begin{lemma}\label{lem:cut-cover}
  Let $Y$ be a closed $3$-manifold and let $\pi\co\tilde{Y}\to Y$ be a
  $\ZZ/2$-cover induced by a $\ZZ$-cover. Then there is a bordered
  $3$-manifold $Y'$ with two boundary components so that:
  \begin{itemize}
  \item $Y=T_{Y'}$, the manifold obtained by gluing the boundary
    components of $Y'$ together
  \item $\tilde{Y}=T_{Y'\cup Y'}$, the manifold obtained by gluing two
    copies of $Y'$ together along their boundary, and
  \item the map $\pi$ is induced by the obvious map $Y'\amalg Y'\to Y'$.
  \end{itemize}
\end{lemma}
\begin{proof}
  With notation as above, suppose that $p=c(q)$. Since $S^1=K(\ZZ,1)$,
  there is a map $f\co Y\to S^1$ so that $q=f^*[S^1]$. Moreover, we
  may assume that $f$ is smooth and that $1\in S^1$ is a regular value
  of $f$. Then the manifold $Y'$ obtained by cutting $Y$ along
  $f^{-1}(1)$ has the desired property.
\end{proof}

\begin{proof}[Proof of Theorem~\ref{thm:honest-dcov}]
  We will suppress the discussion of $\SpinC$-structures, which behave
  similarly to in Theorem~\ref{thm:br-d-cov-sseq-gen}. 

  Let $Y'$ be as in Lemma~\ref{lem:cut-cover} and let $\HD$ be a
  bordered Heegaard diagram for $Y'$, with boundary $-\PMC\amalg
  \PMC$. By Theorem~\ref{thm:hoch-TC},
  \[
  \HFa(\tilde{Y})=\HH_*(\BSDAa(\HD\sos{F(\PMC)}{\cup}{-F(\PMC)}\TC(\PMC))).
  \]
  Let $\widetilde\HD$ denote the result of gluing the boundary
  components of 
  \[
  \HD\sos{F(\PMC)}{\cup}{-F(\PMC)}\TC(\PMC)
  \sos{F(\PMC)}{\cup}{-F(\PMC)}
  \HD\sos{F(\PMC)}{\cup}{-F(\PMC)}\TC(\PMC)
  \]
  together.  On the one hand, the proof of Theorem~\ref{thm:hoch-TC}
  shows that
  \[
  \HH_*\bigl(\BSDAa( \HD\sos{F(\PMC)}{\cup}{-F(\PMC)}\TC(\PMC)
  \sos{F(\PMC)}{\cup}{-F(\PMC)}
  \HD\sos{F(\PMC)}{\cup}{-F(\PMC)}\TC(\PMC)) \bigr)\cong
  \SFH(\widetilde{\HD}).
  \]
  On the other hand, from the topological interpretation of
  $\TC(\PMC)$, $\widetilde{\HD}$ is a sutured Heegaard diagram for
  $\tilde{Y}\setminus (B^3\amalg B^3)$, with one suture on each $S^2$
  boundary component. So, by \cite[Proposition 9.14]{Juhasz06:Sutured},
  \[
  \SFH(\widetilde{\HD})\cong \HFa(\tilde{Y})\otimes H_*(S^1).
  \]
  By Corollary~\ref{cor:is-neutral},
  $\BSDAa(\HD\sos{F(\PMC)}{\cup}{-F(\PMC)}\TC(\PMC))$ is a neutral
  bimodule and so, by   Corollary~\ref{cor:neutral-good-bord},
  $\BSDAa(\HD\sos{F(\PMC)}{\cup}{-F(\PMC)}\TC(\PMC))$ is
  $\pi$-formal. By Proposition~\ref{prop:Alg-is-hom-smooth}, the
  bordered algebras are homologically smooth. So, the result follows
  from Theorem~\ref{thm:hoch-local}.
\end{proof}

\bibliographystyle{hamsalpha}
\bibliography{heegaardfloer}
\end{document}